\documentclass[oneside,abstracton]{scrartcl}
\usepackage[utf8]{inputenc}
\usepackage[T1]{fontenc}
\usepackage{lmodern}
\usepackage[english]{babel}

\usepackage{amsmath}
\usepackage{amsthm}
\usepackage{amssymb}
\usepackage[nobysame,initials]{amsrefs}
\usepackage{enumerate}
\usepackage{mleftright}

\swapnumbers
\mleftright

\theoremstyle{plain}
	\newtheorem{theorem}{Theorem}[section]
	\newtheorem{lemma}[theorem]{Lemma}
	
\theoremstyle{remark}

\newcommand{\cK}	{\mathcal K}
\newcommand{\cP}	{\mathcal P}
\newcommand{\cQ}	{\mathcal Q}
\newcommand{\cR}	{\mathcal R}
\newcommand{\cT}	{\mathcal T}
\newcommand{\diag}[2]	{\ptwotwo {#1} 0 0 {#2}}
\newcommand{\eMiddle}	{e_1^{\odot \frac p 2} \odot e_2^{\odot \frac p 2}}
\newcommand{\GL}[1][n]	{\operatorname{GL}(#1)}
\newcommand{\Id}	{\operatorname{Id}}
\newcommand{\muTr}	{\tilde \mu}
\newcommand{\PL}[1][n]	{\operatorname{GL}^+ (#1)}
\newcommand{\ptwo}[2]	{\begin{pmatrix} #1 \\ #2 \\ \end{pmatrix}}
\newcommand{\ptwotwo}[4]{\begin{pmatrix} #1 & #3 \\ #2 & #4 \\ \end{pmatrix}}

\newcommand{\R}		{\mathbb R}
\newcommand{\C}		{\mathbb C}

\newcommand{\spn}	{\operatorname{span}}
\newcommand{\SL}[1][n]	{\operatorname{SL}(#1)}

\newcommand{\Sym}[2][p]	{\mathrm{Sym}^{#1} ( #2 )}
\newcommand{\Ten}[2][p]	{( #2 )^{\otimes #1}}
\newcommand{\tilt}[1]	{\ptwotwo 1 0 {#1} 1 \cdot}
\newcommand{\ValCo}[2]	{\operatorname{TVal}^{#1}(\mathbb R^{#2})}
\newcommand{\ValCon}[2]	{\operatorname{TVal}_{#1}(\mathbb R^{#2})}
\newcommand{\ValCoZero}[2]	{\operatorname{TVal}^{#1}_0(\mathbb R^{#2})}
\newcommand{\ValCoOne}[2]	{\operatorname{TVal}^{#1}_1(\mathbb R^{#2})}
\newcommand{\ValCoEps}[2]	{\operatorname{TVal}^{#1}_{\varepsilon}(\mathbb R^{#2})}
\newcommand{\ValCoDescr}[3]	{\operatorname{TVal}^{#1}_{#3, F}(\mathbb R^{#2})}
\newcommand{\VL}[1][n]	{\operatorname{SL}^\pm (#1)}

\title{Centro-Affine Tensor Valuations}
\author{Christoph Haberl and Lukas Parapatits}
\date{}

\begin{document}
	\maketitle

	\begin{abstract}
		\noindent
		We completely classify all measurable $\SL$-covariant symmetric tensor valuations on convex
polytopes containing the origin in their interiors.
It is shown that essentially the only examples of such valuations are the moment tensor
and a tensor derived from $L_p$ surface area measures.
This generalizes and unifies earlier results for the scalar, vector and matrix valued case.
 \\[0.5cm]
	\end{abstract}

	\section{Introduction}
		A map $\mu \colon \mathcal{S} \to \langle A, + \rangle$ defined on a collection of sets 
$\mathcal{S}$ with values in an abelian semigroup $\langle A, + \rangle$ is
called a valuation if
\[
	\mu ( P \cup Q ) + \mu ( P \cap Q ) = \mu ( P ) + \mu ( Q )
\]
whenever $P$, $Q$, $P \cup Q$, $P \cap Q \in \mathcal{S}$.

One of the most influential results from the classical Brunn-Minkowski theory is
\mbox{Hadwiger's} classification of continuous rigid motion invariant 
valuations $\mu \colon \cK^n \to \R$. Here, $\cK^n$ denotes the space of convex bodies,
i.e.\ non-empty compact convex subsets of $\R^n$ equipped with the Hausdorff metric.
Hadwiger showed that each such valuation is a linear combination of the intrinsic volumes.
The latter are of basic geometric nature and include volume, surface area, mean width and the Euler Characteristic.

Two fundamental quantities that are not covered by Hadwiger's theorem
are Blaschke's equi-affine and centro-affine surface area.
The latter is not translation invariant and in fact, does not even fit in the framework of the Brunn-Minkowski theory.
It does, however, belong to the so called $L_p$-Brunn-Minkowski theory,
which was shaped by Lutwak \cites{Lut93b, Lut96} in the mid 1990s.
It is based on Firey's $L_p$ addition of convex bodies
containing the origin in their interiors.
The set of all such convex bodies is denoted by $\cK_o^n$.
Since then, this theory has become a central part of modern convex geometry (see \cite{Sch13}*{Chapter 9}). 
The impact of the $L_p$ theory ultimately led to the discovery of an even more general framework:
The Orlicz-Brunn-Minkowski theory (see, e.g., \cites{HabLutYanZha09, garhugwei13, LutYanZha09jdg, 
LutYanZha09badv, ZhuZhouXu14, ZouXiong14, XiJinLeng14, GarHugWeiYe15, LR10, habpar13, Lud09}). 

A characterization of Blaschke's equi-affine and centro-affine surface area
was finally established in a landmark result by Ludwig and Reitzner in \cite{LR10},
where they classified the natural family of Orlicz affine surface areas.
However, one crucial part of the problem remained open
since one of their assumptions was a certain behavior of the maps on convex polytopes.
The first step to bridge this last gap had already been taken by Ludwig \cite{Lud02advval},
but the complete result was only established very recently by the authors \cite{habpar13}:

\begin{theorem}\label{AffHadwiger}
	A map $\mu \colon \cK_o^n\to\R$ is an upper semicontinuous $\SL$-invariant valuation if and only if 
	there exist constants $c_0,c_1, c_2\in\R$ and a function $\varphi\in\textnormal{Conc}(\R_+)$ 
	such that
	\[
		\mu (K)=c_0 \chi(K)+c_1 V(K)+c_2 V(K^*)+\Omega_{\varphi} (K)
	\]
	for all $K\in\cK_o^n$.
\end{theorem}
Here, $\chi$ denotes the Euler Characteristic, $V$ stands for volume, $K^*$ is the polar body
of $K$, and the $\Omega_{\varphi}$ are Orlicz affine surface areas. 
The reader is referred to Section \ref{prelim} and \cite{habpar13} for details. 

This centro-affine Hadwiger theorem has a discrete version for valuations defined on $\cP_o^n$, 
i.e.\ convex polytopes containing the origin in their interiors (see \cites{habpar13, HabParMoments}): 
A map $\mu \colon \cP_o^n\to\R$ is a measurable $\SL$-invariant valuation if and only if there exist constants 
$c_0, c_1, c_2 \in \R$ such that
	\[
		\mu (P)=c_0 \chi(P) + c_1 V(P) + c_2 V(P^*)
	\]
for all $P\in\cP_o^n$.

The aim of the present paper is to generalize this result to tensor valued valuations of arbitrary rank.
Such a generalization to the vector valued case was already established by Ludwig in \cite{Ludwig:moment},
where she characterized the moment vector, i.e.\ the centroid without volume normalization.
In a highly influential article, Ludwig \cite{Lud03} was also able to show the corresponding result for matrix valued valuations.
In both papers, she assumed compatibility with the whole \emph{general} linear group.
A version for the vector valued case that only assumes compatibility with the \emph{special} linear group
was very recently proved by the authors \cite{HabParMoments}.
The present article is the first one to establish a classification for tensor valuations of arbitrary rank in the context
of “centro-affine geometry”.

The study of tensor valuations became the focus of increased attention after Alesker's breakthrough \cite{Alesker99}.
The new techniques developed in this paper enabled him to prove a long sought after characterization
of the rigid motion compatible Minkowski tensors in \cite{Alesker00a}.
In recent years, tensor valuations were studied intensively 
(see, e.g., \cites{AleBerSchu, BerHug15, hugschsch08, Lud03, Lud11, HugSchLocTen, wannerer13, HabParMoments}). 
This is in part due to their applications in morphology and anisotropy analysis of cellular, granular or porous structures
(see, e.g., \cites{BeisArt02, SchrAl2011, SchAl2010, SchrAl2013}).

Let us give two examples of tensor valuations.
Write $\Sym{\R^n} \subseteq \Ten{\R^n}$ for the space of symmetric tensors of rank $p \in \{ 0, 1, \ldots \}$.
The first example is the moment tensor map $M^{p, 0} \colon \cP_o^n \to \Sym{\R^n}$. 
It is defined as
\[
	M^{p, 0}(P) = (n+p) \int_P x^{\odot p}\,dx ,
\]
where integration is with respect to Lebesgue measure and $\odot$ stands for the symmetric tensor product.
Ignoring constant factors, $M^{1, 0}(P)$ is the moment vector
and $M^{2, 0}(P)$ corresponds to the Legendre ellipsoid from classical mechanics.

The second example, $M^{0, p} \colon \cP_o^n \to \Sym{\R^n}$, is given by
\begin{equation}\label{eq: def M^{0,p}}
	M^{0, p}(P) = \int_{S^{n-1}} u^{\odot p}\, dS_p(P,u) .
\end{equation}
Here, $S^{n-1} \subseteq \R^n$ denotes the Euclidean unit sphere and $S_p(P, \cdot)$ is the
$L_p$ surface area measure of $P$ (see Section \ref{prelim} for details).
This tensor vanishes for $p = 1$
and corresponds to the Lutwak-Yang-Zhang ellipsoid from \cite{LutYanZha00duke} for $p = 2$.

We will prove that these are essentially the only examples of tensor valuations 
which are compatible with the $\SL$.
This compatibility is contained in the following definition.
The group $\GL$ acts naturally on $\Ten{\R^n}$	by
\[
	\phi \cdot x = \phi^{\otimes p} (x) 
\]
for all $\phi \in \GL$ and $x \in \Ten{\R^n}$.
A map $\mu \colon \cP_o^n \to \Sym{\R^n}$ is called $\SL$-covariant if
	\[	
		\mu ( \phi P ) = \phi \cdot \mu ( P )
	\]
for all $\phi \in \SL$ and each $P \in \cP_o^n$.	
Moreover, $\mu$ is called measurable if it is Borel measurable.
We are now in a position to state our main result in dimensions greater or equal than three.
	
\begin{theorem}\label{th: MainN}
	Let $p \geq 2$ and $n \geq 3$. 
	A map $\mu \colon \cP_o^n \to \Sym{\R^n}$ is a measurable
	$\SL$-covariant valuation if and only if there exist constants $c_1, c_2 \in \R$ such that
	\[
		\mu(P) = c_1 M^{p, 0}(P) + c_2 M^{0, p}(P^*)
	\]
	for all $P \in \cP_o^n$.
\end{theorem}

The $L_p$ surface area measure appearing in \eqref{eq: def M^{0,p}} is a central notion of the $L_p$-Brunn-Minkowski theory. 
One of the major problems in the field, the so called $L_p$ Minkowski problem, asks which measures are $L_p$ surface area measures 
(see, e.g., \cites{BorLutYanZha, ChoWan06, Lut93b, Zhu2015}). 
Moreover, $L_p$ surface area measures found applications in such diverse fields as affine isoperimetric inequalities
(see, e.g., \cites{CamGro02adv, HS09jdg, LutYanZha00jdg}), Sobolev inequalities (see, e.g., \cites{HSX, CLYZ, LudXiaZha09}),
valuation theory (see, e.g., \cites{Schduke, Hab11, Lud05, Par10a, SchWa15}), 
and information theory (see, e.g., \cites{LutYanZha02duke, Lud11, Lutal2012}).

In \cite{Lut93b}, Lutwak introduced $L_p$ surface area measures in connection with Firey's $L_p$ addition of convex bodies.
In some way or the other, the occurrence of these measures is usually related to this $L_p$ addition.  
We want to emphasize that by Theorem \ref{th: MainN}, $L_p$ surface area measures naturally appear in a completely different context.
This clearly underlines the basic character of these measures. 
For an axiomatic characterization of $L_p$ surface area measures themselves, we refer to \cite{habpar13a}.

The operators occurring in Theorem \ref{th: MainN} are special members of the following family of valuations. 
For $r, s \in \{ 0, 1, \ldots \}$ with $r + s = p$ set
\[
	M^{r, s}(P) = \int_{\partial P} x^{\odot r} \odot u_x^{\odot s} \, h_P^{1-s}(u_x) \, d\mathcal{H}^{n-1}(x) . 
\]
Here, $h_P$ denotes the support function of $P$ and 
integration is with respect to $(n-1)$-dimensional Hausdorff measure over the boundary $\partial P$ of $P$.
The vector $u_x$ is the outer unit normal vector of $P$ at the boundary point $x$. 
Note that $\mathcal{H}^{n-1}$ almost all boundary points have a unique outer unit normal $u_x$.
To the best of our knowledge, these operators have not yet been studied in general.

Except for the extreme cases $M^{p, 0}$ and $M^{0, p} \circ \, ^*$, the members of the above family are not $\SL$-covariant.
However, in the plane they can be modified in a simple way so that they possess this $\SL$-covariance.
Indeed, we denote by $\rho$ the counter-clockwise rotation about an angle of $\frac \pi 2$ and set
\[
	M^{r, s}_{\rho}(P) = \int_{\partial P} x^{\odot r} \odot (\rho u_x)^{\odot s} \, h_P^{1-s}(u_x) \, d\mathcal{H}^1(x) . 
\]
The planar version of Theorem \ref{th: MainN} then reads as follows.

\begin{theorem}\label{th: Main2}
	Let $p \geq 2$.
	A map $\mu \colon \cP_o^2 \to \Sym{\R^2}$ is a measurable
	$\SL[2]$-covariant valuation if and only if there exist constants $c_0, \ldots, c_{p-2}, c_p, c_{p+1} \in \R$ such that
	\[
		\mu(P) = \sum_{\substack{i=0 \\ i \neq p-1}}^p c_i M_{\rho}^{i, p-i}(P) + c_{p+1} \rho \cdot M^{p, 0}(P^*)
	\]
	for all $P \in \cP_o^2$.
\end{theorem}

Denote by $\ValCo{p}{n}$ the vector space of measurable $\SL$-covariant valuations $\mu \colon \cP^n_o \to \Sym{\R^n}$.
As explained above, a complete classification of $\ValCo{0}{n}$ was established in \cites{habpar13, HabParMoments},
whereas the description of $\ValCo{1}{n}$ can be found in \cite{HabParMoments}.
In order to get a complete picture, we finally summarize these results and the main theorems of 
the present article.
\begin{theorem}\label{th: n-dim, SL(n)-covariant}
	For $n \geq 3$ the following holds.
	\begin{itemize}
		\item
		A basis of $\ValCo{0}{n}$ is given by $\chi$, $V$ and $V \circ \, ^*$.

		\item
    A basis of $\ValCo{1}{n}$ is given by $M^{1, 0}$.

		\item
    For $p \geq 2$, a basis of $\ValCo{p}{n}$ is given by $M^{p, 0}$ and $M^{0, p} \circ \, ^*$.
	\end{itemize}
\end{theorem}
As was mentioned before, the planar case is different and needs to be treated seperately.
\begin{theorem}\label{th: 2-dim, SL(2)-covariant}
	For $n = 2$ the following holds.
	\begin{itemize}
		\item
		A basis of $\ValCo{0}{2}$ is given by $\chi$, $V$ and $V \circ \, ^*$.

		\item
    For $p \geq 1$, a basis of $\ValCo{p}{2}$ is given by $M^{i, p-i}_\rho$ for
    $i \in \{ 0, \ldots, p \} \setminus \{ p-1 \}$ and $\rho \cdot M^{p, 0} \circ \, ^*$.
	\end{itemize}
\end{theorem}
In this paper, we will actually provide a unified proof of Theorems \ref{th: n-dim, SL(n)-covariant}
and \ref{th: 2-dim, SL(2)-covariant}.
Therefore, we also provide new proofs of the results in \cite{habpar13} and \cite{HabParMoments} which fit into the general
context of tensor valuations with arbitrary rank.

	\section{Notation and preliminary results}\label{prelim}
		For later use, we collect in this section notation and basic facts.
Well known results about convex bodies are stated without references.
We refer the reader to the excellent books of 
Gardner \cite{Gar95}, Gruber \cite{Gruber:CDG}, and Schneider \cite{Sch13}
for more information.

Let us begin with two one-dimensional facts.
The first one is the solution to Cauchy's functional equation. 
As is well known, the only measurable functions $f \colon \R \to \R$ which satisfy
\[
	f(x + y) = f(x) + f(y)
\]
for all $x, y \in \R$ are the linear ones.
The same holds for functions $f \colon (0, \infty) \to \R$ and $f \colon \R^n \to \R$.
The second one is a version of Vandermonde's identity,
\begin{equation}\label{eq: Vandermonde}
	\sum_{j=0}^i \binom {-\frac p 2} {i-j} \binom {\frac p 2} {j} = 0
\end{equation}
for $i \geq 1$.
This follows from the equality $(1 + x)^{-\frac p 2}(1 + x)^{\frac p 2} = 1$
by comparing coefficients of the Taylor expansions of the involved functions.

Now, we turn towards higher dimensions.
The space $\R^n$, $n \geq 1$, will be equipped with the standard inner product and the norm induced by it.
Denote by $e_1,\ldots, e_n \in \R^n$ the canonical basis vectors and
write $S^{n-1}$ for the set of all unit vectors with respect to this norm.

Throughout this paper, we fix the standard basis of $\Ten{\R^n}$ induced by the canonical basis vectors $e_1, \ldots, e_n$.
For tensors $x_1, \ldots, x_p \in \Ten{\R^n}$, their symmetric tensor product is defined as 
\[
	x_1 \odot \cdots \odot x_p = \frac {1}{p!} 
	\sum_{\sigma \in \mathfrak{S_p}} x_{\sigma(1)} \otimes \cdots \otimes x_{\sigma(p)} ,
\]
where $\mathfrak{S_p}$ denotes the symmetric group of $\{1, 2, \ldots, p\}$.
Note that the normalization is chosen in such a way that $x \odot \cdots \odot x = x \otimes \cdots \otimes x$.
Let $K \in \Ten{\R^2}$ and $\alpha \in \{1, 2\}^p$ be a multiindex.
If $\phi \in \GL[2]$, then the action of $\phi$ on $K$ can be written as
\begin{equation}\label{eq: action on tensors}
	\phi \cdot K = \sum_{\alpha} \sum_{\beta} K_\beta \phi_{\alpha_1 \beta_1} \cdots \phi_{\alpha_p \beta_p} 
	               e_{\alpha_1} \otimes \cdots \otimes e_{\alpha_p} .
\end{equation}
Here, $K_\beta$ denote the coefficients of $K$ with respect to the basis we fixed before.
Multiindices will be viewed as being equipped with their standard partial order.
Therefore, we immediately arrive at	
	\begin{equation}\label{eq: tilt expanded}
  	\left[ \tilt z K \right]_{\alpha} = \sum_{\beta \geq \alpha} 
  	K_{\beta} z^{\vert\{i \colon \beta_i = 2\}\vert - \vert\{i \colon \alpha_i = 2\}\vert}.
  \end{equation}
For the space $\Sym{\R^2}$ of symmetric tensors we fix the basis 
\[
	e_1^{\odot p-i} \odot e_2^{\odot i} , \qquad i = 0, \ldots, p .
\]
Let $K \in \Sym{\R^2}$.
The coordinates of $K$ with respect to this basis are denoted by $K_i$.
For even $p$ and a $\phi \in \GL[2]$ we therefore have
\begin{equation}\label{eq: pth component middle term}
	\left[ \phi \cdot e_1^{\odot \frac p 2} \odot e_2^{\odot \frac p 2} \right]_p 
	= ( \phi_{2 \! 1} \phi_{2 2} )^{\frac p 2} .
\end{equation}
Relation \eqref{eq: tilt expanded} and a straightforward computation prove
\begin{equation}\label{eq: tilt expanded symmetric}
 	\left[ \tilt z K \right]_i = \sum_{j=i}^p \binom j i K_j z^{j-i} .
\end{equation}
This in turn yields
  \begin{equation}\label{eq: tilt expanded symmetric transpose}
  	\left[ \ptwotwo 1 z 0 1 \cdot K \right]_i = \sum_{j=0}^i \binom {p-j} {p-i} K_j z^{i-j} .
  \end{equation}
Let us briefly discuss a tensor version of Cauchy's functional equation.
Let $F \colon \R^n \to \Ten{\R^n}$ be a measurable function with
\[
	F(x + y) = F(x) + F(y)
\]
for all $x, y \in \R^n$.
Using the scalar Cauchy equation, it is not hard to show that the component functions of $F$ are linear.
Interpreting tensors as multilinear maps, it therefore follows that there exists an $\tilde F \in \Ten[p+1]{\R^n}$ 
such that
\begin{equation}\label{eq: tensor cauchy}
	F(x) (v_1, \ldots, v_p) = \tilde F (v_1, \ldots, v_p, x)
\end{equation}
for all $v_1, \ldots, v_p, x \in \R^n$.
In other words, $F$ can be interpreted as an element of $\Ten[p+1]{\R^n}$.

Next, we collect some facts about tensor integrals.
As usual, integrals over tensors are defined componentwise.
Thus, a straightforward calculation in combination with \eqref{eq: action on tensors}
proves for a continuous function $F \colon [a,b] \to \Ten{\R^n}$ that 
	\begin{equation}\label{eq: integral compatible SL(n)}
		\int_a^b \phi \cdot F(x) \,dx = \phi \cdot \int_a^b F(x) \,dx 
	\end{equation}
for all $\phi \in \GL$.
If $F \colon \R \to \Ten{\R^n}$ is continuous, then one can check the symmetry of its images 
by looking at certain integrals. Indeed, by the componentwise definition of the integral and 
differentiation we have
	\begin{equation}\label{eq: int symmetric}
		F(x) \in \Sym{\R^n} \quad \textnormal{for all } x \in \R
		\Longleftrightarrow 
		\int_0^x F(z) \,dz \in \Sym{\R^n} \quad \textnormal{for all } x \in \R .
	\end{equation}
We conclude our treatment of integrals with the following injectivity type properties.
For $K \in \Ten{\R^2}$, we infer from \eqref{eq: tilt expanded} that 
	\begin{equation}\label{eq: integral tilt kernel}
	  \int_0^x \tilt z K \,dz = 0 \quad \textnormal{for some } x \in \R \setminus \{ 0 \}
	  \quad \Longleftrightarrow \quad K = 0 .
	\end{equation}
Hence,
\[
	K \mapsto \int_0^1 \tilt z K \ dz
\]
is a linear isomorphism on $\Ten{\R^2}$.
A variant of this implication is
	\begin{equation}\label{eq: integral tilt kernel transpose}
	  \int_0^x \ptwotwo 1 {-z} 0 1 K \,dz = 0 \quad \textnormal{for some } x \in \R \setminus \{ 0 \}
	  \quad \Longleftrightarrow \quad K = 0 .
	\end{equation}

Let a convex polytope $P \in \cP_o^n$ be given.
In the next paragraph we recall some basic geometric quantities associated with $P$.
The first example is the support function $h_P$.
This is the function $h_P \colon \R^n \to \R$ defined by
\[
	h_P(x) = \max \{ x \cdot y: y \in P \}.
\]
The polar body $P^*$ of $P$ is given by
\[
	P^* = \{ x \in \R^n: x \cdot y \leq 1 \text{ for all } y \in P \} .
\]
Note that for each $\phi \in \GL$ we have
\begin{equation}\label{eq: polarity contravariant}
	(\phi P)^* = \phi^{-t} P^*,
\end{equation}
where $\phi^{-t}$ denotes the inverse of the transpose of $\phi$.
In particular,
\begin{equation}\label{eq: polar homogeneity}
	(\lambda P)^* = \lambda^{-1} P^*
\end{equation}
for all positive $\lambda$.
We define the polarity map $ ^* \colon \cP_o^n \to \cP_o^n$ 
as the function which assigns to each polytope its polar body.
It is well known that this map is a homeomorph involution.
The surface area measure $S(P, \cdot)$ is defined for each Borel set $\omega \subseteq S^{n-1}$ as
\[
	S(P,\omega) = \mathcal{H}^{n-1}\{ x \in \partial P: 
	\exists \text{ an outer unit normal } u_x \text{ at } x \text{ which belongs to } \omega \}.
\] 
Surface area measures have their centroid at the origin, i.e.
\begin{equation}\label{eq: surface centroid}
	\int_{S^{n-1}} u \ dS(P, u) = 0
\end{equation}
for all $P \in \cP_o^n$.
The $L_p$ surface area measure $S_p(P, \cdot)$ is given by
\[
	S_p(P, \omega) = \int_{\omega} h_P^{1-p}(u) \ dS(P, u). 
\]

Next, we will generalize the concept of $\SL$-covariance a little bit.
We write $\VL$ for the set of linear maps having determinant either $1$ or $-1$.
Let $\varepsilon \in \{0,1\}$ and $G \subseteq \VL$ be given.
A map $\mu \colon \cP^n_o \to \Sym{\R^n}$ is said to be $G$-$\varepsilon$-covariant 
or $\varepsilon$-covariant with respect to $G$ if
	\[
		\mu (\phi P) = (\det \phi)^{\varepsilon} \phi \cdot \mu (P)
	\]
for every $P \in \cP^n_o$ and each $\phi \in G$.
In order to simplify the notation in the sequel, we write $\ValCoEps{p}{n}$
for the vector space of measurable $\VL$-$\varepsilon$-covariant valuations.

Let $\mu \in \ValCo{p}{n}$.
Choose a $\theta \in \VL \setminus \SL$.
For all $P \in \cP^n_o$ define
\[
	\mu^0 (P) = \frac 1 2 \left( \mu (P) + \theta \cdot \mu (\theta^{-1} P) \right)
\]
and
\[
	\mu^1 (P) = \frac 1 2 \left( \mu (P) - \theta \cdot \mu (\theta^{-1} P) \right).
\]
The $\SL$-covariance of $\mu$ implies that these definitions do not depend on the choice of $\theta$.
Clearly, $\mu^0$ and $\mu^1$ are measurable valuations and $\mu = \mu^0 + \mu^1$.
Moreover, it is easy to see that $\mu^0 \in \ValCoZero{p}{n}$ and $\mu^1 \in \ValCoOne{p}{n}$.
Hence,
\begin{equation}\label{eq: Val direct sum}
	\ValCo{p}{n} = \ValCoZero{p}{n} \oplus \ValCoOne{p}{n}.
\end{equation}

The convex hull of a set $A \subseteq \R^n$ is written as $[A]$.
In the context of double pyramids, the following symbols will usually have a fixed meaning.
The letters $a, b, c, d$ will denote positive real numbers
with associated line segments $I := [-a e_1, b e_1]$ and $J := [-c e_n, d e_n]$, respectively.
The letters $x, y$ will denote elements of $\R^{n-1}$.
In particular, for $n = 2$ we have $J = [-c e_2, d e_2]$ and $x,y \in \R$.
The letter $B$ will denote an element of $\cP^{n-1}_o$.
For $n = 2$, we say that $a, b, c, d, x, y$ form a double pyramid if
\[
	\left[ I, -c \ptwo x 1, d \ptwo y 1 \right] \cap e_2^\perp = I
\]
and for $n \geq 3$, we say that $B, c, d, x, y$ form a double pyramid if
\[
	\left[ B, -c \ptwo x 1, d \ptwo y 1 \right] \cap e_n^\perp = B .
\]
If $x = y = 0$, then we call the double pyramid straight.
The set of double pyramids will be denoted by $\cR^n$
and the set of straight double pyramids by $\cQ^n$.
In \cite{Lud02advval}, Ludwig proved that if a real valued valuation $\mu \colon \cP^n_o \to \R$ 
vanishes on all $\SL[n]$-images of elements in $\cR^n$, then it vanishes on $\cP_o^n$.
A componentwise application of this fact yields the following result.

\begin{theorem}\label{th: determined by values on SL(n)(R^n)}
	Let $n \geq 2$.
	Suppose that $\mu \colon \cP^n_o \to \Ten{\R^n}$ is a valuation which vanishes on all $\SL[n]$-images of elements in $\cR^n$.
	Then $\mu$ vanishes everywhere.
\end{theorem}
Clearly, for $n = 2$, straight double pyramids can be split into two triangles having one side contained
in $e_2^{\bot}$. The set of such straight triangles 
	\[
		[I, -c e_2] \quad \text{ and } \quad [I, d e_2]	
	\]
is denoted by $\cT^2$.

We need an explicit description of $M_\rho^{i, p-i}$ on straight double pyramids.
We start with the following calculation.
Note that in the next two lemmas we use the common convention that
\begin{equation}\label{eq: convention for binomial coefficient}
	\binom i j = 0 \quad \text{if $j < 0$ or $j > i$} .
\end{equation}

\begin{lemma}\label{le: M^{j,p-j} help}
	For $b, c > 0$ and $i \in \{ 0, \ldots, p\}$, we have
	\begin{equation} \label{eq: int b c}
		(i+1) (bc)^{1-p+i} \int_0^1 {\ptwo {bt} {-c(1-t)}}^{\odot i} \odot {\ptwo b c}^{\odot p-i} \ dt = \\
		\sum_{l = 0}^p m_{i,l} b^{1+i-l} c^{1-p+i+l} e_1^{\odot p-l} \odot e_2^{\odot l} ,
	\end{equation}
	where
	\[
		m_{i,l} = \binom {p-i-1} l + (-1)^i \binom {p-i-1} {l-i-1}
	\]
	for $i \neq p$ and
	\[
		m_{p,l} = (-1)^l .
	\]
	In particular, for $i \neq p$, $m_{i,l} = 0$ if $p-i \leq l \leq i$.
\end{lemma}

\begin{proof}
	Define $L$ as the left hand side of \eqref{eq: int b c}.
	Writing the first tensor product in coordinates,
	using well known results for the Beta function,
	and writing the second tensor product in coordinates,
	we calculate
	\begin{align*}
		L
		&= (i+1) \sum_{j = 0}^i \binom i j (-1)^j b^{1-p+2i-j} c^{1-p+i+j} \int_0^1 t^{i-j} (1-t)^j \ dt \
		e_1^{\odot i-j} \odot e_2^{\odot j} \odot {\ptwo b c}^{\odot p-i} \\
		&= \sum_{j = 0}^i (-1)^j b^{1-p+2i-j} c^{1-p+i+j}
		e_1^{\odot i-j} \odot e_2^{\odot j} \odot {\ptwo b c}^{\odot p-i} \\
		&= \sum_{j = 0}^i \sum_{k = 0}^{p-i} \binom {p-i} k (-1)^j b^{1+i-j-k} c^{1-p+i+j+k}
		e_1^{\odot p-j-k} \odot e_2^{\odot j+k} .
	\end{align*}
	Summing first over $l = j+k$ and keeping convention \eqref{eq: convention for binomial coefficient}
	for the binomal coefficient in mind, this becomes
	\[
		L = \sum_{l = 0}^p \sum_{j = 0}^i \binom {p-i} {l-j} (-1)^j b^{1+i-l} c^{1-p+i+l}
		e_1^{\odot p-l} \odot e_2^{\odot l} .
	\]
	For $i = p$, we clearly have
	\[
		\sum_{j = 0}^p \binom 0 {l-j} (-1)^j = (-1)^l .
	\]
	Assume $i \neq p$.
	A well known formula for an alternating sum of binomal coefficients states
	\[
		\sum_{j = 0}^l \binom {p-i} j (-1)^j = (-1)^l \binom {p-i-1} l .
	\]
	Note that $p-i \geq 1$ and that we again use convention \eqref{eq: convention for binomial coefficient}
	for the binomal coefficient from above.
	Hence,
	\begin{align*}
		\sum_{j = 0}^i \binom {p-i} {l-j} (-1)^j
		&= (-1)^l \sum_{j = l-i}^l \binom {p-i} j (-1)^j \\
		&= \binom {p-i-1} l + (-1)^i \binom {p-i-1} {l-i-1} .
	\end{align*}
\end{proof}

With the aid of this result, we can now calculate $M_\rho^{i, p-i}$ on
straight double pyramids.

\begin{lemma}\label{le: M on crosspolytopes}
	For $i \in \{ 0, \ldots, p\}$, we have
	\begin{multline*}
		M_\rho^{i, p-i}[I,J]
		= \frac 1 {i+1} \sum_{l = 0}^p m_{i,l} \left[
		(-1)^{i+l} a^{1+i-l} c^{1-p+i+l}
		+ b^{1+i-l} c^{1-p+i+l}
		\right. \\ \left.
		+ (-1)^p a^{1+i-l} d^{1-p+i+l}
		+ (-1)^{p+i+l} b^{1+i-l} d^{1-p+i+l}
		\right] e_1^{\odot p-l} \odot e_2^{\odot l}
	\end{multline*}
	for all $a, b, c, d > 0$, where $m_{i,l}$ is defined as in Lemma \ref{le: M^{j,p-j} help}.
	In particular, for $p \geq 1$, $M_\rho^{p-1, 1}[I,J] = 0$.
\end{lemma}

\begin{proof}
	Clearly, the double pyramid $[I,J]$ has four edges.
	Hence, the defining integral of $M_\rho^{i, p-i}$ can be split into four integrals along these line segments.
	Let us consider the edge $[b e_1, -c e_2]$.
	Using the parametrization 
	\[
		\gamma(t) = t b e_1 - (1-t)c e_2, \quad t \in [0, 1],
	\]
	an elementary calculation shows that 
	\[
		\int_{[b e_1, -c e_2]} x^{\odot i} \odot (\rho u_x)^{\odot p-i} \ d\mathcal H^1(x)
	\]
	equals, up to a factor of $i+1$, the integral considered in Lemma \ref{le: M^{j,p-j} help}.
	Similar observations are true for the other three edges.
	Summing the expressions from Lemma \ref{le: M^{j,p-j} help} for all edges yields the desired result.
	Finally, for $i \neq p$, note that the terms for $l = i+1$ as well as $l = p-i-1$ cancel out.
\end{proof}

Let $B = [I,J]$. Note that $B^* = [- a^{-1}, b^{-1}] \times [- c^{-1}, d^{-1}]$. 
Thus,
\[
	\left[ M^{p,0}(B^*) \right]_p = 
	\frac{1}{p+1} \left( a^{-1} + b^{-1} \right)\left( d^{-p-1} + (-1)^p c^{-p-1} \right).
\]
In combination with the last lemma, we see that
\begin{equation}\label{eq: M(B) neq 0}
	\left[ M_\rho^{i,p-i} (B) \right]_0 , \quad
	i \in \{0,\ldots,p\}\backslash\{p-1\} , \quad \text{and} \quad
	\left[\rho \cdot M^{p,0}(B^*)\right]_0
\end{equation}
do not vanish for all double pyramids $B$.
It follows immediately from their definition, that the $M_\rho^{i,p-i}$ are
homogeneous, i.e.
\begin{equation}\label{eq: M homogeneous 2 dim}
	M_\rho^{i,p-i}( \lambda P )= \lambda^{2-p+2i} M_\rho^{i,p-i}(P)
\end{equation}
for all $P \in \cP_o^2$ and $\lambda > 0$.
Combining the last two facts and \eqref{eq: polar homogeneity}, it is easy to see that the family 
\[
	M_\rho^{i,p-i}, i \in \{0,\ldots,p\}\backslash\{p-1\}, \quad \rho \cdot M^{p,0} \circ \, ^* 
\]
is linearly independent.

We also state a few simple facts about $M^{i,p-i}$ in dimension $n \geq 2$ for later reference.
Clearly, these maps are homogeneous,
\begin{equation}\label{eq: M homogeneous n dim}
	M^{i,p-i}( \lambda P )= \lambda^{n-p+2i} M^{i,p-i}(P)
\end{equation}
for all $P \in \cP_o^n$ and $\lambda > 0$.
The $e_1^{\odot p}$-coordinates of
\begin{equation}\label{eq: M(B) neq 0, n dim}
	M^{p, 0}(B^*) \quad \text{and} \quad M^{0, p}(B)
\end{equation}
do not vanish for all crosspolytopes $B$ except for $M^{0,1}(B)$.
To see this, one can look at crosspolytopes that are sufficiently asymmetric with respect to $e_1^\perp$.
Combining the last two facts and \eqref{eq: polar homogeneity}, it is easy to see that
\[
	M^{p, 0} \circ \, ^* \quad \text{and} \quad M^{0, p}
\]
are linearly independent.

Next, we show that $M_\rho^{i,p-i}$, for $n = 2$, and $M^{i,p-i}$, for $n \geq 2$, are valuations.
The easiest way to see this is to write them as integrals over the support measure $\Lambda_{n-1}$ (see \cite{Sch13}*{Chapter 4}).
For the latter, we have
\[
	M^{i,p-i}(P) = 2 \int_{\R \times S^{n-1}} x^{\odot i} \odot u^{\odot p-i} \, (x \cdot u)^{1-p+i} \, d\Lambda_{n-1}(P, (x,u))
\]
for all $P \in \cP^n_o$.
Now, the valuation property, the covariance properties, as well as the continuity of these maps
follow from similar properties of the support measure $\Lambda_{n-1}$.
In particular,
\begin{equation}\label{eq: M_rho signum covariant}
	M_{\rho}^{i,p-i}( \phi P )= (\det \phi)^{p-i} \ \phi \cdot M_{\rho}^{i,p-i}(P) ,
\end{equation}
for all $P \in \cP_o^2$ and $\phi \in \VL[2]$.
Furthermore,
\begin{equation}\label{eq: M covariant, n dim}
	M^{p,0}( \phi P )= \phi \cdot M^{p,0}(P) \quad \text{and} \quad
	M^{0,p}( (\phi P)^* ) = \phi \cdot M^{0,p}( P^* )
\end{equation}
for all $P \in \cP_o^n$ and $\phi \in \VL[n]$.

	\section{Proof of the Main Results}
		%------------------------------------------------------------------------------%
\subsection{The $1$-dimensional case}
%------------------------------------------------------------------------------%
We aim at a description of $\ValCo{p}{1}$.
Note that $\left( \R^1 \right)^{\otimes p}$ is always isomorphic to $\R$. 
Moreover, an $\VL[1]$-$\varepsilon$-covariant map is either even or odd.
So it suffices to classify even and odd valuations $\mu \colon \cP^1_o \to \R$, respectively.
Such classifications were already established in \cites{habpar13,HabParMoments} and are stated below.
Let us begin with the even case.

\begin{theorem}\label{th: dim 1 even}
	Suppose that $\mu \colon \cP^1_o \to \R$ is a measurable valuation.
	Then $\mu$ is even if and only if there exists a measurable function $F \colon (0, \infty) \to \R$ such that
	\[
		\mu[-a, b] = F(a) + F(b)
	\]
	for all $a, b > 0$.
	Moreover, $F(a) = \frac 1 2 \mu[-a,a]$.
\end{theorem}

For homogeneous valuations even more can be said. 
In fact, the function $F$ from the above theorem can be described explicitely.

\begin{theorem}\label{th: dim 1 even homogeneous}
  Suppose that $\mu \colon \cP^1_o \to \R$ is a measurable valuation.
	Then $\mu$ is even and homogeneous of degree $r \in \R$ if and only if there exists a constant $c \in \R$ such that
	\[
		\mu[-a, b] = c\,( a^r + b^r )
	\]
	for all $a, b > 0$.
\end{theorem}

Next, we state the corresponding classifications for odd valuations.

\begin{theorem}\label{th: dim 1 odd}
	Suppose that $\mu \colon \cP^1_o \to \R$ is a measurable valuation.
	Then $\mu$ is odd if and only if there exists a measurable function $F \colon (0, \infty) \to \R$ such that
	\[
		\mu[-a, b] = F(b) - F(a)
	\]
	for all $a, b > 0$.
	Moreover, $F(a) = \mu[-1,a] + c$ for some constant $c \in \R$.
\end{theorem}

As before, an immediate consequence of Theorem \ref{th: dim 1 odd} is a classification of odd homogeneous valuations.

\begin{theorem}\label{th: dim 1 odd homogeneous}
	Suppose that $\mu \colon \cP^1_o \to \R$ is a measurable valuation.
	Then $\mu$ is odd and homogeneous of degree $r \in \R \setminus \{ 0 \}$ if and only if there exists a constant $c \in \R$ such that
	\[
		\mu[-a, b] = c\,( b^r - a^r )
	\]
	for all $a, b > 0$.
	
	The valuation $\mu$ is odd and homogeneous of degree $0$ if and only if there exists a constant $c \in \R$ such that
	\[
		\mu[-a, b] = c\,[ \ln(b) - \ln(a) ]
	\]
	for all $a, b > 0$.
\end{theorem}

We remark that every valuation $\mu \colon \cP^1_o \to \R$ can be written as the sum of an even and an odd valuation. 
Therefore, the above theorems yield a classification of all measurable valuations $\mu \colon \cP^1_o \to \R$.

%------------------------------------------------------------------------------%
\subsection{The $2$-dimensional case}

\subsubsection{Some tensor equations}
 %------------------------------------------------------------------------------%
We begin by solving a sheared version of Cauchy's functional equation for tensors.

\begin{lemma}\label{le: tensor Cauchy}
	Suppose that $G \colon \R \to \Ten{\R^2}$ is a measurable function.
	Then $G$ satisfies
	\begin{equation}\label{eq: sheared Cauchy}
		G(x+y) = G(x) + \tilt x G(y)
	\end{equation}
	for all $x, y \in \R$ if and only if there exists a tensor $K \in \Ten{\R^2}$ such that
	\begin{equation}\label{eq: sheared Cauchy result}
		G(x) = \int_0^x \tilt z K \ dz .
	\end{equation}
	Moreover, if $G$ has symmetric images, then $K \in \Sym{\R^2}$.
	Furthermore, the same results hold if $G$ is only defined on $(0, \infty)$.
\end{lemma}
\begin{proof}
	An elementary calculation combined with \eqref{eq: integral compatible SL(n)} proves 
	\begin{align}\label{eq: sheared Cauchy check}
		\int_0^{x+y} \tilt z K \ dz
		&= \int_0^x \tilt z K \ dz + \int_x^{x+y} \tilt z K \ dz \nonumber \\ 
		&= \int_0^x \tilt z K \ dz + \int_0^y \tilt {z+x} K \ dz \nonumber \\
		&= \int_0^x \tilt z K \ dz + \tilt x \int_0^y \tilt z K \ dz 
	\end{align}
	for each $K \in \Ten{\R^2}$.
	So each $G$ defined by \eqref{eq: sheared Cauchy result} satisfies \eqref{eq: sheared Cauchy}.
	
	Now, let $G$ be a solution of \eqref{eq: sheared Cauchy}. 
	By \eqref{eq: integral tilt kernel} we can find a $K \in \Ten{\R^2}$ such that
	\[
		H(x) := G(x) - \int_0^x \tilt z K \ dz
	\]
	satisfies $H(1) = 0$.
	It remains to prove that $H(x) = 0$ for all $x \in \R$.
	We will show by induction that $H_\alpha = 0$ for all $\alpha \in \{ 1, 2 \}^p$.
	Assume that $H_\beta = 0$ for all $\beta > \alpha$, which is trivially true for $\alpha = (2, \ldots, 2)$.
	Equation \eqref{eq: sheared Cauchy} is clearly a linear functional equation.
	Hence, by the definition of $H$ and \eqref{eq: sheared Cauchy check}, $H$ satisfies \eqref{eq: sheared Cauchy}.
	So \eqref{eq: tilt expanded} and the induction assumption yield
	\[
		H_\alpha (x+y) = H_\alpha (x) + H_\alpha (y).
	\]
	Thus, $H_\alpha$ satisfies Cauchy's functional equation.
	Since $H_\alpha$ is measurable and $H_\alpha (1) = 0$, it follows that $H_\alpha = 0$.
	
	We still have to prove the assertion about symmetric tensors.
	By assumption,
	\[
		G(x) = \int_0^x \tilt z K \ dz \in \Sym{\R^2} 
	\]
	for all $x \in \R$.
	So from \eqref{eq: int symmetric} we infer that
	\[
		\tilt z K \in \Sym{\R^2}
	\]
	for all $z \in \R$.
	Since the above matrix is invertible, also $K \in \Sym{\R^2}$.
	The proof for $G \colon (0, \infty) \to \Ten{\R^2}$ is exactly the same.
\end{proof}

The next result is a slighlty different version of Lemma \ref{le: tensor Cauchy}.

\begin{lemma}\label{le: tensor Cauchy, contravariant}
	Suppose that $G \colon \R \to \Ten{\R^2}$ is a measurable function.
	Then $G$ satisfies
	\begin{equation}\label{eq: sheared Cauchy, contravariant}
		G(x+y) = G(x) + \ptwotwo 1 {-x} 0 1 \cdot G(y)
	\end{equation}
	for all $x, y \in \R$ if and only if there exists a tensor $K \in \Ten{\R^2}$ such that
	\[
		G(x) = \int_0^x \ptwotwo 1 {-z} 0 1 \cdot K \ dz .
	\]
	Moreover, if $G$ has symmetric images, then $K \in \Sym{\R^2}$.
	Furthermore, the same results hold if $G$ is only defined on $(0, \infty)$.
\end{lemma}
\begin{proof}
	Define a function $H \colon \R \to \Ten{\R^2}$ by
	\[
		H(x) = \ptwotwo 0 {-1} 1 0 \cdot G(x). 
	\]
	Then $G$ satisfies \eqref{eq: sheared Cauchy, contravariant} if and only if $H$ satisfies \eqref{eq: sheared Cauchy}.
	By Lemma \ref{le: tensor Cauchy} this happens precisely if there exists a $J \in \Ten{\R^2}$ with
	\[
		H(x) = \int_0^x \tilt z J \ dz .
	\]
	Rewriting $H$ in terms of $G$ and setting
	\[
		K = \ptwotwo 0 1 {-1} 0 \cdot J
	\]
	concludes the proof.
\end{proof}

Using Lemma \ref{le: tensor Cauchy} we now establish the solution of a more intricate functional equation.
In fact, this functional equation will be crucial for the proof of our main theorem.

\begin{lemma}\label{le: solve homogeneous equation}
	Let $\varepsilon \in \{0,1\}$ and $F \colon (0, \infty) \to \Ten{\R^2}$ be a measurable function.
	The function $F$ satisfies
	\begin{equation}\label{eq: F(t)}
		F(t) =
		\ptwotwo 1 {\frac 1 {st}} 0 1 \cdot F \left( \frac {st} {s+1} \right)
		+ (-1)^\varepsilon \ptwotwo s {\frac 1 t} t 0 \cdot F \left( \frac t {s+1} \right)
	\end{equation}
	and
	\begin{equation}\label{eq: F mirrored}
		(-1)^\varepsilon \ptwotwo 0 {\frac 1 s} s 0 \cdot F(s) = F(s)
	\end{equation}
	for all $s, t > 0$ if and only if there exists a tensor $K \in \Ten{\R^2}$ with
	\begin{equation}\label{eq: K even/odd}
		(-1)^\varepsilon \diag {-1} 1 \cdot K = K
	\end{equation}
	such that
	\begin{equation}\label{eq: F solution integral}
		F(x) = \int_0^x \ptwotwo 1 {- \frac 1 x} z {1 - \frac z x} \cdot K \ dz , \quad x > 0 .
	\end{equation}
\end{lemma}
\begin{proof}
  Define a function $G \colon (0, \infty) \to \Ten{\R^2}$ by
	\[
		G(x) = \ptwotwo 1 {\frac 1 x} 0 1 \cdot F(x) .
	\]	
	We will first show that $F$ satisfies \eqref{eq: F(t)} and \eqref{eq: F mirrored} for all $s, t > 0$ if and only if $G$ satisfies
	\begin{equation}\label{eq: G(x+y)}
		G(x+y) = G(x) + \ptwotwo 1 0 x 1 \cdot G(y) 
	\end{equation}
	and
	\begin{equation}\label{eq: G mirrored}
		(-1)^\varepsilon \ptwotwo {-1} 0 x 1 \cdot G(x) = G(x) 
	\end{equation}
	for all $x, y >0$.
	In order to do so, we consider the coordinate transformation $s = \frac x y$ and $t = x+y$.
	Multiplying \eqref{eq: F(t)} by 
	\[
		\ptwotwo 1 {\frac 1 t} 0 1
	\]
	and rewriting the resulting equation in terms of $x$ and $y$ shows that \eqref{eq: F(t)} is equivalent to
	\[
		\ptwotwo 1 {\frac 1 {x+y}} 0 1 \cdot F(x+y) =
		\ptwotwo 1 {\frac 1 x} 0 1 \cdot F(x)
		+ (-1)^\varepsilon \ptwotwo {\frac x y} {\frac 1 y} {x+y} 1 \cdot F(y) .
	\] 
	By the definition of $G$, the last equation holds precisely if 
	\begin{equation}\label{eq: F(t) substituted 1}
		G(x+y) = G(x) + (-1)^\varepsilon \ptwotwo {-1} 0 {x+y} 1 \cdot G(y) .
	\end{equation}
	Clearly, $F$ satisfies \eqref{eq: F mirrored} if and only if $G$ satisfies \eqref{eq: G mirrored}.
  In combination with \eqref{eq: F(t) substituted 1} this proves the desired equivalence
  \[
  	F \textnormal{ satisfies \eqref{eq: F(t)} and \eqref{eq: F mirrored} for all } s, t > 0
  	\Longleftrightarrow 
  	G \textnormal{ satisfies \eqref{eq: G(x+y)} and \eqref{eq: G mirrored} for all } x, y >0 .
  \] 	
 	From Lemma \ref{le: tensor Cauchy} we infer that $G$ solves \eqref{eq: G(x+y)} if and only if
	\[
		G(x) = \int_0^x \tilt z K \ dz
	\]
	for some tensor $K \in \Ten{\R^2}$.
	By \eqref{eq: integral compatible SL(n)} and a substitution we obtain
	\[
		\ptwotwo {-1} 0 x 1 \cdot \int_0^x \tilt z K \ dz =
		\int_0^x \ptwotwo {-1} 0 {x-z} 1 \cdot K \ dz =
		\int_0^x \ptwotwo {-1} 0 z 1 \cdot K \ dz .
	\]
	Using \eqref{eq: integral tilt kernel}, we see that $G$ satisfies \eqref{eq: G mirrored} if and only if
	\[
		(-1)^\varepsilon \diag {-1} 1 \cdot K = K .
	\]
	Rewriting $G$ in terms of $F$ concludes the proof.
\end{proof}

\subsubsection{Splitting over pyramids}
%------------------------------------------------------------------------------%
Let $\mu \in \ValCoEps{p}{2}$.
We say that $\mu$ splits over pyramids if the following three conditions hold.
First, there is a measurable map $\muTr \colon \cT^2 \to \Sym{\R^2}$ with
\[
	\mu [I, J] = \muTr [I, -c e_2] + \muTr [I, d e_2]
\]
for all $a, b, c, d > 0$.
Recall that by our notation convention we set $I = [-a e_1, b e_1]$ and $J = [-c e_2, d e_2]$.
Second,  for all $c, d > 0$ the maps
\[
	I \mapsto \muTr [I, -c e_2] \quad \text{and} \quad I \mapsto \muTr [I, d e_2]
\]
are valuations on $\cP_o^1$.
Third, $\muTr$ is $\varepsilon$-covariant with respect to the transformations
\[
	\diag {-1} 1 \quad \text{and} \quad \diag 1 {-1} .
\]

In Subsection \ref{sec: planar symmetric case} we will construct splittings explicitely.
However, for now we assume that such a splitting exists.

Clearly, a double pyramid can be divided into two separately tilted triangles. 
The idea of the next lemma is to compare the value of $\mu$ on a double pyramid 
with the values of a splitting on these triangles.
As it turns out, the error term in this comparison has suprisingly nice properties.

In the sequel, we will repeatedly use the following obvious fact.
If $a, b > 0$ and $x, y \in \R$ are given, then for sufficiently small $c, d > 0$ the numbers $a, b, c, d, x, y$ form a double pyramid.

\begin{lemma}\label{le: existence F^I}
  Let $\mu \in \ValCoEps{p}{2}$.
	If $\mu$ splits over pyramids,
	then there exists a family of functions $F^I \colon \R^2 \to \Sym{\R^2}$ such that
	\begin{equation}\label{eq: def F^I}
		\mu \left[ I, -c \ptwo x 1, d \ptwo y 1 \right] =
		\tilt x \muTr [I, -c e_2] + \tilt y \muTr [I, d e_2] + F^I(x, y)
	\end{equation}
	for all $a, b, c, d > 0$ and $x, y \in \R$ which form a double pyramid.
	Furthermore, each $F^I$ satisfies
	\begin{equation}\label{eq: F^I(0, y-x)}
		F^I(x, y) = \tilt x F^I(0, y-x)
	\end{equation}
	and
	\begin{equation}\label{eq: F^I(0, x+y)}
		F^I(0, x+y) = F^I(0, x) + \tilt x F^I(0, y) 
	\end{equation}
	for all $x, y \in \R$.
\end{lemma}
\begin{proof}
	Let $a, b > 0$ and $x, y \in \R$ be given. 
	Choose $c, d > 0$ such that $a, b, c, d, x, y$ form a double pyramid.
	For sufficiently small $r > 0$ the valuation property implies
	\begin{multline*}
		\mu \left[ I, -c \ptwo x 1, d \ptwo y 1 \right] + \mu \left[ I, -r \ptwo y 1, r \ptwo y 1 \right] = \\
		\mu \left[ I, -c \ptwo x 1, r \ptwo y 1 \right] + \mu \left[ I, -r \ptwo y 1, d \ptwo y 1 \right] .
	\end{multline*}
	Since $\mu$ is $\VL[2]$-$\varepsilon$-covariant and splits over pyramids, we have
	\begin{multline*}
		\mu \left[ I, -c \ptwo x 1, d \ptwo y 1 \right] - \tilt y \muTr [I, d e_2] = \\
		\mu \left[ I, -c \ptwo x 1, r \ptwo y 1 \right] - \tilt y \muTr [I, r e_2] .
	\end{multline*}
	In other words, the expression on the left hand side is independent of $d$.
	Similarly,
	\[
		\mu \left[ I, -c \ptwo x 1, d \ptwo y 1 \right] - \tilt x \muTr [I, -c e_2]
	\]
	is independent of $c$.
	Consequently, the term
	\[
		\mu \left[ I, -c \ptwo x 1, d \ptwo y 1 \right] - \tilt x \muTr [I, -c e_2] - \tilt y \muTr [I, d e_2]
	\]
	is independent of $c$ and $d$. 
	This proves the existence of functions $F^I$ which satisfy \eqref{eq: def F^I}.

	Next, we establish relation \eqref{eq: F^I(0, y-x)}.
	By the $\VL[2]$-$\varepsilon$-covariance of $\mu$ and equation \eqref{eq: def F^I} we have
	\begin{multline*}
		\mu \left[ I, -c \ptwo x 1, d \ptwo y 1 \right] \\
		\begin{aligned}
			&= \tilt x \mu \left[ I, -c e_2, d \ptwo {y-x} 1 \right] \\
			&= \tilt x \left( \muTr [I, -c e_2] + \tilt {y-x} \muTr [I, d e_2] + F^I(0, y-x) \right) \\
			&= \tilt x \muTr [I, -c e_2] + \tilt y \muTr [I, d e_2] + \tilt x F^I(0, y-x) .
		\end{aligned}
	\end{multline*}
	A glance at \eqref{eq: def F^I} quickly yields \eqref{eq: F^I(0, y-x)}.

	It remains to show \eqref{eq: F^I(0, x+y)}. 
	For sufficiently small $r > 0$ the valuation property implies
	\[
		\mu \left[ I, -c \ptwo x 1, d \ptwo y 1 \right] + \mu \left[ I, -r e_2, r e_2 \right] =
		\mu \left[ I, -c \ptwo x 1, r e_2 \right] + \mu \left[ I, -r e_2, d \ptwo y 1 \right] .
	\]
	Using \eqref{eq: def F^I} and the fact that $\mu$ splits over pyramids give
	\[
		F^I(x, y) = F^I(x, 0) + F^I(0, y) .
	\]
	With the aid of \eqref{eq: F^I(0, y-x)} we finally arrive at
	\[
		\tilt x F^I(0, y-x) = \tilt x F^I(0, -x) + F^I(0, y) .
	\]
	Replacing $x$ by $-x$ immediately yields \eqref{eq: F^I(0, x+y)}.
\end{proof}

Now, we are going to use the solution of the sheared Cauchy equation \eqref{eq: sheared Cauchy}
to get a more explicit representation for the $F^I$.

\begin{lemma}\label{le: mu of double pyramid}
	Let $\mu \in \ValCoEps{p}{2}$.
	If $\mu$ splits over pyramids, then there exists a family of tensors $K^I \in \Sym{\R^2}$ such that
	\begin{equation}\label{eq: mu of double pyramid}
		\mu \left[ I, -c \ptwo x 1, d \ptwo y 1 \right] =
		\tilt x \muTr [I, -c e_2] + \tilt y \muTr [I, d e_2] + \int_x^y \tilt z K^I \ dz
	\end{equation}
	for all $a, b, c, d > 0$ and $x, y \in \R$ which form a double pyramid.
\end{lemma}
\begin{proof}
	Fix an interval $I$ and let $F^I$ be the function from Lemma \ref{le: existence F^I}. 
	By \eqref{eq: F^I(0, x+y)} and Lemma \ref{le: tensor Cauchy} there exists a $K^I \in \Sym{\R^2}$ with
	\begin{equation}\label{eq: f with K}
		F^I(0, x) = \int_0^x \tilt z K^I \ dz
	\end{equation}
	for all $x \in \R$.
	From relations \eqref{eq: F^I(0, y-x)}, \eqref{eq: f with K}, a substitution, and \eqref{eq: integral compatible SL(n)} we obtain
	\begin{align*}
		F^I(x, y)
		&= \tilt x F^I(0, y-x) \\
		&= \tilt x \int_0^{y-x} \tilt z K^I \ dz \\
		&= \tilt x \int_x^y \tilt {z-x} K^I \ dz \\
		&= \int_x^y \tilt z K^I \ dz .
	\end{align*}
	Thus, equation \eqref{eq: def F^I} immediately implies \eqref{eq: mu of double pyramid}.
\end{proof}

\subsubsection{The main results} \label{sec: planar symmetric case}
%------------------------------------------------------------------------------%
After these preparations we will now prove our description of $\ValCo{p}{2}$.
We start by showing that every $\mu \in \ValCoEps{p}{2}$ splits over pyramids.
Recall that we fixed a basis
\[
	e_1^{\odot p-i} \odot e_2^{\odot i} , \qquad i = 0, \ldots, p
\]
and that the $i$-th coordinate of $\mu$ with respect to this basis is denoted by $\mu_i$.

\begin{lemma}\label{le: mu of straight double pyramid, VL(2)-covariant}
	Each $\mu \in \ValCoEps{p}{2}$ splits over pyramids.
	Furthermore, there exists a splitting with the following two properties:
	For $i \in \{ 0, \ldots, p \}$ and $a, b, d > 0$,
	\begin{equation}\label{eq: splitting homogeneous}
		\muTr_i [I, d e_2] = d^{2i-p} \muTr_i [d I, e_2]
	\end{equation}
	if $i+\varepsilon$ is even and
	\begin{equation}\label{eq: splitting zero}
		\muTr_i [I, e_2] = 0
	\end{equation}
	if $i+\varepsilon$ is odd.
\end{lemma}
\begin{proof}
	We begin with the simple observation that $J \mapsto \mu_i [I, J]$ is a measurable valuation.
	By the $\VL[2]$-$\varepsilon$-covariance of $\mu$ we therefore obtain
	\begin{equation}\label{eq: mu even odd}
		\mu_i [I, -J] = (-1)^{i+\varepsilon} \mu_i [I, J] .
	\end{equation}	
	
	Let $c, d > 0$. Define a map $\muTr \colon \cT^2 \to \Sym{\R^2}$ componentwise by
  \[
		\muTr_i [I, -c e_2] = \tfrac 1 2\, \mu_i [I, -c e_2, c e_2] , \qquad
		\muTr_i [I, d e_2]  = \tfrac 1 2\, \mu_i [I, -d e_2, d e_2] 
	\]
	for even $i+\varepsilon$ and 
	\[
		\muTr_i [I, -c e_2] = - \mu_i [I, - e_2, c e_2] , \qquad
		\muTr_i [I, d e_2]  = \mu_i [I, - e_2, d e_2] 
	\]	
	for odd $i+\varepsilon$.
	
	Next, we will show that relations \eqref{eq: splitting homogeneous} and \eqref{eq: splitting zero} hold.
	If $i+\varepsilon$ is even, then the definition of $\muTr$ and the $\VL[2]$-$\varepsilon$-covariance of $\mu$ yield
	\begin{align*}
		\muTr_i [I, d e_2]
		&= \tfrac 1 2\, \mu_i [I, -d e_2, d e_2] \\
		&= \frac 1 2\, \left[ \diag {\frac 1 d } d \mu [d I, -e_2, e_2] \right]_i \\
		&= \tfrac 1 2\, d^{2i-p} \mu_i [d I, -e_2, e_2] \\
		&= d^{2i-p} \muTr_i [dI, e_2].
	\end{align*}
	Hence, relation \eqref{eq: splitting homogeneous} holds.	
	If $i+\varepsilon$ is odd, then the definition of $\muTr$ and \eqref{eq: mu even odd} give
	\[
		\muTr_i [I, e_2] = \mu_i [I, -e_2, e_2] = 0.
	\]
  So $\muTr$ satisfies \eqref{eq: splitting zero}.
	
	It remains to show that $\muTr$ is actually a splitting. 
	First, suppose that $i+\varepsilon$ is even. From \eqref{eq: mu even odd} and Theorem \ref{th: dim 1 even} we infer
	\begin{align*}
		\mu_i [I, J] 
		&= \tfrac 1 2\, \mu_i [I, -c e_2, c e_2] + \tfrac 1 2\, \mu_i [I, -d e_2, d e_2] \\
		&= \muTr_i [I, -c e_2] + \muTr_i [I, d e_2] .
	\end{align*}
	Second, let $i+\varepsilon$ be odd. 
	By relation \eqref{eq: mu even odd} and Theorem \ref{th: dim 1 odd} we obtain
	\begin{align*}
		\mu_i [I, J] 
		&= \mu_i [I, -e_2, d e_2] - \mu_i [I, -e_2, c e_2]  \\
		&= \muTr_i [I, -c e_2] + \muTr_i [I, d e_2] .
	\end{align*}
	So $\muTr$ has the additivity property required for a splitting.
	From the definition of $\muTr$ and the respective properties of $\mu$
	it follows easily that $\muTr$ possesses the desired valuation and covariance property.
\end{proof}

Recall from Lemma \ref{le: mu of double pyramid} that a splitting can be used to describe $\mu$ on double pyramids.
Our next result reveals that the above splitting can be modified in such a way that 
it is determined by a function $F \colon (0, \infty) \to \Sym{\R^2}$.
Moreover, the error term in Lemma \ref{le: mu of double pyramid} can be calculated explicitely.

\begin{lemma}\label{le: mu of double pyramid, VL(2)-covariant}
	Let $\mu \in \ValCoEps{p}{2}$.
	There exist a measurable function $F \colon (0, \infty) \to \Sym{\R^2}$
	and a constant $k \in \R$ such that
	\begin{align}\label{eq: mu of double pyramid, VL(2)-covariant}
		\mu \left[ I, -c \ptwo x 1, d \ptwo y 1 \right] =
		& \ptwotwo {\frac 1 c} 0 {cx} {c} \cdot \left( (-1)^p F(ac) + (-1)^\varepsilon \diag 1 {-1} \cdot F(bc) \right) \nonumber \\
		& {} + \ptwotwo {\frac 1 d} 0 {dy} {d} \cdot \left( (-1)^\varepsilon \diag {-1} 1 \cdot F(ad) + F(bd) \right) \nonumber \\
		& {} + k \left( (-1)^{p+1} a^{-p-2} + b^{-p-2} \right) \int_x^y \tilt z e_2^{\odot p} \ dz
	\end{align}
	for all $a, b, c, d > 0$ and $x, y \in \R$ which form a double pyramid.
	Furthermore, $k = 0$ if $p+\varepsilon$ is odd.
\end{lemma}
\begin{proof} 
  Let $\muTr$ be the splitting from Lemma \ref{le: mu of straight double pyramid, VL(2)-covariant}
  and suppose that $a, b, c, d > 0$ and $x, y \in \R$ form a double pyramid.
	By Lemma \ref{le: mu of double pyramid}, equation \eqref{eq: tilt expanded symmetric}, 
	a basic fact about binomial coefficients, and an index shift we obtain
	\begin{multline*}
		\mu_i \left[ I, -c \ptwo x 1, d \ptwo y 1 \right] \\
		\begin{aligned}
			&= \left[ \tilt x \muTr [I, -c e_2] + \tilt y \muTr [I, d e_2] + \int_x^y \tilt z K^I \ dz \right]_i \\ 
			&= \sum_{j = i}^p \binom j i \left( x^{j-i} \muTr_j [I, -c e_2] + y^{j-i} \muTr_j [I, d e_2] \right)
			   + \sum_{j = i}^p \binom j i \frac {y^{j-i+1} - x^{j-i+1}} {j-i+1} K^I_j \\
			&= \sum_{j = i}^p \binom j i \left( x^{j-i} \muTr_j [I, -c e_2] + y^{j-i} \muTr_j [I, d e_2] \right)
			   + \sum_{j = i}^p \binom {j+1} i \frac {y^{j-i+1} - x^{j-i+1}} {j+1} K^I_j \\
			&= \sum_{j = i}^p \binom j i \left( x^{j-i} \muTr_j [I, -c e_2] + y^{j-i} \muTr_j [I, d e_2] \right) 
	    	 + \sum_{j = i + 1}^{p + 1} \binom {j} i \frac {y^{j-i} - x^{j-i}} {j} K^I_{j - 1}
		\end{aligned}
	\end{multline*}
	for all $i \in \{ 0, \ldots, p \}$.
	As usual, we write $J = [-c e_2, d e_2]$. Rearranging sums in the last formula gives
	\begin{multline}\label{eq: mu of double pyramid in coordinates}
		\mu_i \left[ I, -c \ptwo x 1, d \ptwo y 1 \right] =
		\mu_i [I, J] +
		\sum_{j = i+1}^p \binom j i x^{j-i} \left( \muTr_j [I, -c e_2] - \frac {K^I_{j-1}} j \right) \\
		{} + \sum_{j = i+1}^p \binom j i y^{j-i} \left( \muTr_j [I, d e_2] + \frac {K^I_{j-1}} j \right) +
		\binom {p+1} i \frac {y^{p+1-i} - x^{p+1-i}} {p+1} K^I_p.
	\end{multline}
  
  Assume that we know the following.
  First, there exists a constant $k \in \R$ with
	\begin{equation}\label{eq: K^I_p}
		K^I_p = k \left( (-1)^{p+1} a^{-p-2} + b^{-p-2} \right) .
	\end{equation}
	Second, there exist measurable functions $F_j \colon (0,\infty) \to \R$, $j \in \{ 0, \ldots, p \}$, such that
	\begin{equation}\label{eq: splitting c + K^I = F unified}
		\muTr_j [I, -c e_2] - \frac {K^I_{j-1}} j = c^{2j-p} \left( (-1)^p F_j(ac) + (-1)^{j+\varepsilon} F_j(bc) \right) 
	\end{equation}
  and
  \begin{equation}\label{eq: splitting d + K^I = F unified}
		\muTr_j [I, d e_2] + \frac {K^I_{j-1}} j = d^{2j-p} \left( (-1)^{p+j+\varepsilon} F_j(ad) + F_j(bd) \right) .
	\end{equation}
	for $j \neq 0$.
	Third, for these functions also the equality
	\begin{equation}\label{eq: mu F on straight pyramids}
		\mu_i[I, J] = 
		c^{2i-p} \left( (-1)^p F_i(ac) + (-1)^{i+\varepsilon} F_i(bc) \right) + 
		d^{2i-p} \left( (-1)^{p+i+\varepsilon} F_i(ad) + F_i(bd) \right) 
	\end{equation}
	holds for all $i \in \{ 0, \ldots, p \}$. 
	Under these assumptions, we can plug \eqref{eq: K^I_p}, \eqref{eq: splitting c + K^I = F unified}, 
	\eqref{eq: splitting d + K^I = F unified}, and \eqref{eq: mu F on straight pyramids} 
	into \eqref{eq: mu of double pyramid in coordinates}.
	This results in 
 	\begin{align*}
		\mu_i \left[ I, -c \ptwo x 1, d \ptwo y 1 \right] =&
		\sum_{j = i}^p \binom j i x^{j-i} c^{2j-p} \left( (-1)^p F_j(ac) + (-1)^{j+\varepsilon} F_j(bc) \right) \\
		& {} + \sum_{j = i}^p \binom j i y^{j-i} d^{2j-p} \left( (-1)^{p+j+\varepsilon} F_j(ad) + F_j(bd) \right) \\
		& {} + k \left( (-1)^{p+1} a^{-p-2} + b^{-p-2} \right) \binom {p+1} i \frac {y^{p+1-i} - x^{p+1-i}} {p+1}.
	\end{align*}
	But by \eqref{eq: tilt expanded symmetric} this is, in coordinates, precisely what we want to show.
	
	It remains to prove \eqref{eq: K^I_p}, \eqref{eq: splitting c + K^I = F unified}, 
	\eqref{eq: splitting d + K^I = F unified}, and \eqref{eq: mu F on straight pyramids}.
	In order to do so, fix $a, b > 0$ and $x, y \in \R$. 
  The $\VL[2]$-$\varepsilon$-covariance of $\mu$ with respect to the reflection at $e_1^{\bot}$ and the origin yields
	\begin{equation}\label{eq: mu reflection y-axis}
		\mu \left[ -I, -c \ptwo x 1, d \ptwo y 1 \right] =
		(-1)^\varepsilon \diag {-1} {1} \cdot \mu \left[ I, -c \ptwo {-x} 1, d \ptwo {-y} 1 \right]
	\end{equation}
	and
	\begin{equation}\label{eq: mu reflection origin}
		\mu \left[ -I, -c \ptwo x 1, d \ptwo y 1 \right] =
		(-1)^p \mu \left[ I, -d \ptwo y 1, c \ptwo x 1 \right] .
	\end{equation} 
	For sufficiently small $c, d >0$ all arguments of $\mu$ in \eqref{eq: mu reflection y-axis} and \eqref{eq: mu reflection origin}
	are double pyramids.
	Hence, we can apply representation \eqref{eq: mu of double pyramid} to \eqref{eq: mu reflection y-axis} 
	and \eqref{eq: mu reflection origin}. Thus,
	\begin{multline*}
		\tilt x \muTr [-I, -c e_2] + \tilt y \muTr [-I, d e_2] + \int_x^y \tilt z K^{-I} \ dz = \\
		(-1)^\varepsilon \ptwotwo {-1} 0 x 1 \cdot \muTr [I, -c e_2] + (-1)^\varepsilon \ptwotwo {-1} 0 y 1 \cdot \muTr [I, d e_2] +
		(-1)^\varepsilon \int_{-x}^{-y} \ptwotwo {-1} 0 {-z} 1 \cdot K^I \ dz
	\end{multline*}
	and
	\begin{multline*}
		\tilt x \muTr [-I, -c e_2] + \tilt y \muTr [-I, d e_2] + \int_x^y \tilt z K^{-I} \ dz = \\
		(-1)^p \tilt y \muTr [I, -d e_2] + (-1)^p \tilt x \muTr [I, c e_2] +
		(-1)^p \int_y^x \tilt z K^I \ dz .
	\end{multline*}
	The terms involving $\muTr$ in the above equations cancel due to the $\varepsilon$-covariance of splittings.
	Applying elementary transformations to the remaining integrals therefore proves 
	\[
		\int_x^y \tilt z K^{-I} \ dz = (-1)^{\varepsilon+1} \int_x^y \ptwotwo {-1} 0 z 1 \cdot K^I \ dz
	\]
	and
	\[
		\int_x^y \tilt z K^{-I} \ dz = (-1)^{p+1} \int_x^y \tilt z K^I \ dz .
	\]
	Thus, the injectivity property \eqref{eq: integral tilt kernel} implies
	\begin{equation}\label{eq: K diag}
		K^{-I} = (-1)^{\varepsilon+1} \diag {-1} 1 \cdot K^I
	\end{equation}
	and
	\begin{equation}\label{eq: K (-1)^p}
		K^{-I} = (-1)^{p+1} K^I .
	\end{equation}
	Let $j \in \{ 0, \ldots, p \}$. Writing the last equation componentwise shows
	\begin{equation}\label{eq: K^I_j, j even}
		K^{-I}_j = (-1)^{p+1} K^I_j .
	\end{equation}
	This relation implies that $I \mapsto K^I_j$ is either even or odd.
	Bearing \eqref{eq: integral tilt kernel} and \eqref{eq: mu of double pyramid} in mind,
	it is easy to see that $I \mapsto K^I_j$ is also a measurable valuation.
	If we combine \eqref{eq: K diag} and \eqref{eq: K (-1)^p}, then we have in addition
	\begin{equation}\label{eq: K^I_j, j odd}
		K^I_j = 0 \qquad \textnormal{for } j+\varepsilon \textnormal{ odd} .
	\end{equation}	
	
	Next, fix $a, b, d > 0$. 
	From the $\VL[2]$-$\varepsilon$-covariance of $\mu$ we deduce
	\[
		\mu \left[ d I, - \ptwo x 1,  \ptwo y 1 \right] =
		\diag d {\frac 1 d} \cdot \mu \left[ I, -d \ptwo {\frac x {d^2}} 1, d \ptwo {\frac y {d^2}} 1 \right] .
	\]
	In particular, the 0-th component of $\mu$ satisfies
	\begin{equation}\label{eq: 0th component}
		\mu_0 \left[ d I, - \ptwo x 1, \ptwo y 1 \right] =
		d^p \mu_0 \left[ I, -d \ptwo {\frac x {d^2}} 1, d \ptwo {\frac y {d^2}} 1 \right] .
	\end{equation}
	Note that both arguments of $\mu$ in the last equation are double pyramids for sufficiently small $x$ and $y$.
	So we can apply \eqref{eq: mu of double pyramid in coordinates} to \eqref{eq: 0th component}.
	Therefore, both sides of \eqref{eq: 0th component} are polynomials in $x$ and $y$ on a small rectangle around the origin.
	Comparing the coefficients of $y^{p+1}$ yields
	\[
		\frac {K^{d I}_p} {p+1} = d^{-p-2} \frac {K^I_p} {p+1} .
	\]
	In other words, $I \mapsto K^I_p$ is $(-p-2)$-homogeneous.
	Recall from \eqref{eq: K^I_j, j even} that $I \mapsto K^I_p$ is a measurable valuation which is either even or odd.
	So Theorems \ref{th: dim 1 even homogeneous} and \ref{th: dim 1 odd homogeneous}
	prove the existence of a constant $k \in \R$ such that \eqref{eq: K^I_p} holds.
	From \eqref{eq: K^I_j, j odd} we also know that $K^I_p$ vanishes if $p+\varepsilon$ is odd. 
	Thus, $k = 0$ if $p+\varepsilon$ is odd. 
	
	Let $j \in \{ 1, \ldots, p \}$ and suppose that $j+\varepsilon$ is odd. 
	In order to get information on $K^I_{j-1}$ we proceed similar as before.
	Indeed, comparing the coefficients of $y^j$ in \eqref{eq: 0th component} yields
	\[
		\muTr_j [d I, e_2] + \frac {K^{d I}_{j-1}} j  =
		d^{p-2j} \left( \muTr_j [I, d e_2] + \frac {K^I_{j-1}} j \right) .
	\]	
	Since $j+\varepsilon$ is odd, we can use \eqref{eq: splitting zero} to get
	\begin{equation}\label{eq: splitting + K^I}
		\muTr_j [I, d e_2] + \frac {K^I_{j-1}} j = d^{2j-p} \frac {K^{d I}_{j-1}} j .
	\end{equation}
	Again, recall that $I \mapsto K^I_{j-1}$ is a measurable valuation which is either even or odd.
	So Theorems \ref{th: dim 1 even} and \ref{th: dim 1 odd} prove the existence of a
	measurable function $F_j \colon (0, \infty) \to \R$ such that
	\[
		\frac {K^I_{j-1}} j = (-1)^{p+1} F_j(a) + F_j(b) .
	\]
	Plugging this into the right hand side of \eqref{eq: splitting + K^I} gives 
	\[
		\muTr_j [I, d e_2] + \frac {K^I_{j-1}} j = d^{2j-p} \left( (-1)^{p+1} F_j(ad) + F_j(bd) \right) .
	\]
  A corresponding formula also exists for triangles contained in the lower half plane. 
	Indeed, combining the last equality with the $\varepsilon$-covariance of splittings yields
	\[
		\muTr_j [I, -c e_2] - \frac {K^I_{j-1}} j = c^{2j-p} \left( (-1)^p F_j(ac) - F_j(bc) \right) . 
	\]
	Therefore, we established \eqref{eq: splitting c + K^I = F unified} and \eqref{eq: splitting d + K^I = F unified}
	for odd $j+\varepsilon$.
	
	Next, let $j+\varepsilon$ be even.
  The $\varepsilon$-covariance of splittings shows
	\[
		\muTr_j[-I, - e_2] = (-1)^p \muTr_j[I, - e_2].
	\]
	Consequently, the map $I \mapsto \muTr_j[I, - e_2]$ is either even or odd. 
	Moreover, it is a measurable valuation by definition.
	Now Theorems \ref{th: dim 1 even} and \ref{th: dim 1 odd} show that
	there exists a measurable function $F_j \colon (0, \infty) \to \R$ such that
	\[
		\muTr_j [I, - e_2] = (-1)^p F_j(a) + F_j(b) .
	\]
	From the $\varepsilon$-covariance of splittings and \eqref{eq: splitting homogeneous} we infer
	\[
		\muTr_j [I, - c e_2] = c^{2j-p} \muTr_j [cI, - e_2],
	\]
	which results in
	\[
		\muTr_j [I, - c e_2] = c^{2j-p} \Big( (-1)^p F_j(ac) + F_j(bc) \Big).
	\]	
	Using again the $\varepsilon$-covariance of splittings, we also have the following representation 
	\[
		\muTr_j [I, d e_2] = d^{2j-p} \Big( (-1)^p F_j(ad) + F_j(bd) \Big) 
	\]
	for	triangles contained in the upper half plane.
	Recall from \eqref{eq: K^I_j, j odd} that $K^I_{j-1} = 0$. 
	Therefore, the last two equations prove \eqref{eq: splitting c + K^I = F unified} and \eqref{eq: splitting d + K^I = F unified}
	for even $j+\varepsilon$.
	
	Finally, we have to show the validity of \eqref{eq: mu F on straight pyramids}.
	For $i \geq 1$, this is a simple consequence of adding \eqref{eq: splitting c + K^I = F unified} 
	and \eqref{eq: splitting d + K^I = F unified}.
	The case $i = 0$ remains.
	For $p \neq 0$, set $F_0 (x) = (-1)^{\varepsilon} x^p F_p(x)$ and note that the $\VL[2]$-$\varepsilon$-covariance of $\mu$ implies
	\[
		\mu_0[I, J] = (-1)^{\varepsilon} \mu_p[\tilde J, \tilde I] ,
	\]
	where $\tilde I := [-a e_2, b e_2]$ and $\tilde J := [-c e_1, d e_1]$.
	Since we already have \eqref{eq: mu F on straight pyramids} for $i = p$, this and the definition
	of $F_0$ prove \eqref{eq: mu F on straight pyramids} for $i = 0.$
	This last step does not work for $p = 0$.
	For $p = 0$ and $\varepsilon = 0$ one can easily deduce \eqref{eq: mu F on straight pyramids} as in \cite{habpar13}*{Lemma 3.5}.
	For $p = 0$ and $\varepsilon = 1$ the situation is a little different.
	We refer to \cite{habpar13}*{Lemma 3.6} and \cite{HabParMoments}*{Theorem 2.3} for a proof that \eqref{eq: mu F on straight pyramids}
	holds with $F = 0$ in this case.
\end{proof}

Let $\mu \in \ValCoEps{p}{2}$. 
The last lemma shows that, up to an additive term, $\mu$ is determined by a function $F$.
This motivates the following definition.
We say that a measurable function $F \colon (0, \infty) \to \Sym{\R^2}$ describes $\mu$ if
\begin{align}\label{eq: mu describable by a function F}
	\mu \left[ I, -c \ptwo x 1, d \ptwo y 1 \right] = &
	(-1)^p \ptwotwo {\frac 1 c} 0 {cx} c \cdot F(ac) +
	(-1)^\varepsilon \ptwotwo {\frac 1 c} 0 {-cx} {-c} \cdot F(bc) \nonumber \\
	& {} + (-1)^\varepsilon \ptwotwo {- \frac 1 d} 0 {dy} d \cdot F(ad) +
	\ptwotwo {\frac 1 d} 0 {dy} d \cdot F(bd)
\end{align}
for all $a, b, c, d > 0$ and $x, y \in \R$ which form a double pyramid.
In coordinates \eqref{eq: mu describable by a function F} reads as
 	\begin{align}\label{eq: F describes mu in coordinates}
		\mu_i \left[ I, -c \ptwo x 1, d \ptwo y 1 \right] =&
		\sum_{j = i}^p \binom j i x^{j-i} c^{2j-p} \left( (-1)^p F_j(ac) + (-1)^{j+\varepsilon} F_j(bc) \right) \nonumber \\
		& {} + \sum_{j = i}^p \binom j i y^{j-i} d^{2j-p} \left( (-1)^{p+j+\varepsilon} F_j(ad) + F_j(bd) \right) , 
	\end{align}
which can be easily seen using \eqref{eq: tilt expanded symmetric}.

In general, there is some freedom in the choice of the describing function $F$.
So it makes sense to single out a particular $F$ with useful additional properties.
This will be done in the next lemma.

\begin{lemma}\label{le: ptwotwo 0 frac 1 a a 0 F(a) = F(a)}
	Let $\mu \in \ValCoEps{p}{2}$.
	If $\mu$ can be described by some measurable function,
	then there exists a measurable $\tilde F \colon (0, \infty) \to \Sym{\R^2}$ and a constant $k \in \R$ such that
	\begin{equation}\label{eq: F + ln describe mu}
		\tilde F + k \ln \, \eMiddle
	\end{equation}
	also describes $\mu$ and 
	\begin{equation}\label{eq: ptwotwo 0 frac 1 a a 0 tilde F(a) = tilde F(a)}
		(-1)^\varepsilon \ptwotwo 0 {\frac 1 a} a 0 \cdot \tilde F(a) = \tilde F(a)
	\end{equation}
	holds for all $a > 0$.
	
	If $p$ as well as $\frac p 2$ are even and $\varepsilon = 1$, then $\tilde F_{\frac p 2} = 0$. 
	In all other cases we have $k = 0$.
\end{lemma}

There is a slight abuse of notation in \eqref{eq: F + ln describe mu}, since we write
\[
	k \ln \, \eMiddle
\]
even in cases where the above tensor might not be defined, i.e.\ for odd $p$. 
However, as stated in Lemma \ref{le: ptwotwo 0 frac 1 a a 0 F(a) = F(a)} in all such cases $k = 0$ holds.
This actually means that the corresponding term does not show up at all and hence $\tilde F$ itself describes $\mu$.
This notation has the advantage that we can avoid distinctions of cases in the sequel. 
	
\begin{proof}	
	Denote by $F$ the measurable function which describes $\mu$.
	Let $j \in \{ 0, \ldots, p \}$ and assume that $p$ has a different parity than $j+\varepsilon$.
	Then it follows from \eqref{eq: F describes mu in coordinates} that a constant can be added to $F_j$ 
	without changing \eqref{eq: mu describable by a function F}.
	Therefore, without loss of generality, we assume that $F_j(1) = 0$ for all such $j$.

	By the $\VL[2]$-$\varepsilon$-covariance of $\mu$ we have
	\[
		\mu [-a e_1, b e_1, -c e_2, d e_2] = (-1)^\varepsilon \ptwotwo 0 1 1 0 \mu [-c e_1, d e_1, -a e_2, b e_2]
	\]
	for all $a, b, c, d > 0$.
	If we plug representation \eqref{eq: mu describable by a function F} into the above terms we obtain
	\begin{multline}\label{eq: mu[I, J] = mu[J ,I] for F}
		(-1)^p \diag {\frac 1 c} c \cdot F(ac) +
		(-1)^\varepsilon \diag {\frac 1 c} {-c} \cdot F(bc) +
		(-1)^\varepsilon \diag {- \frac 1 d} d \cdot F(ad) +
		\diag {\frac 1 d} d \cdot F(bd)
		\\ = \\
		(-1)^{p+\varepsilon} \ptwotwo 0 {\frac 1 a} a 0 \cdot F(ac) +
		\ptwotwo 0 {\frac 1 a} {-a} 0 \cdot F(ad) +
		\ptwotwo 0 {- \frac 1 b} b 0 \cdot F(bc) +
		(-1)^\varepsilon \ptwotwo 0 {\frac 1 b} b 0 \cdot F(bd) .
	\end{multline}
	
	We start with the case where $p$ is even.
	For $j \in \{ 0, \ldots, p \}$ such that $j + \varepsilon$ is even we choose $b = a$ and $c = d = 1$
	in \eqref{eq: mu[I, J] = mu[J ,I] for F}. 
	Thus,
	\begin{equation}\label{eq: F_j symmetric}
		4 F_j(a) = (-1)^\varepsilon 4 a^{p-2j} F_{p-j}(a) .
	\end{equation}
	For $j \in \{ 0, \ldots, p \}$ such that $j + \varepsilon$ is odd we set $b = d = 1$
	in \eqref{eq: mu[I, J] = mu[J ,I] for F}. Together with the assumption $F_j(1) = 0$ for such $j$ we obtain
	\begin{equation}\label{eq: F for special values}
		c^{2j-p} F_j(ac) - c^{2j-p} F_j(c) - F_j(a) =
		(-1)^\varepsilon \left( a^{p-2j} F_{p-j}(ac) - a^{p-2j} F_{p-j}(a) - F_{p-j}(c) \right) .
	\end{equation}
	Define measurable functions $G_j \colon (0, \infty) \to \R$ by
	\[
		G_j(a) = a^{2j-p} F_j(a) - (-1)^\varepsilon F_{p-j}(a) .
	\]
	An immediate consequence of this definition is the fact that
	\begin{equation}\label{eq: G_p-j is almost G_j}
		G_{p-j}(a) = (-1)^{\varepsilon+1} a^{p-2j} G_j(a) .
	\end{equation}
	Using the definition of $G_j$, equation \eqref{eq: F for special values} can be written as
	\begin{equation}\label{eq: Cauchy G}
		G_j(ac) = G_j(a) + a^{2j-p} G_j(c) .
	\end{equation}
	First, suppose that $j \neq \tfrac p 2$. 
	The left hand side of this equation is symmetric in $a$ and $c$.
	So interchanging the roles of $a$ and $c$ yields
	\[
		G_j(a) + a^{2j-p} G_j(c) = G_j(c) + c^{2j-p} G_j(a).
	\]
	Therefore, we arrive at
	\[	 
		G_j(a) = \frac {G_j(c)} {1-c^{2j-p}} \left( 1-a^{2j-p} \right) 
	\]
	for all $a, c > 0$. 
	If we choose $c = 2$ and define constants $g_j \in \R$ by
	\[
	g_j = \frac {G_j(2)} {1-2^{2j-p}},
	\]
	then we have
	\[
		G_j(a) = g_j \left( 1-a^{2j-p} \right) .
	\]
	Plugging the definition of $G_j$ into this relation shows
	\[
		a^{2j-p} F_j(a) - (-1)^\varepsilon F_{p-j}(a) = g_j \left(1-a^{2j-p} \right) .
	\]
	From \eqref{eq: G_p-j is almost G_j} we infer that $g_{p-j} = (-1)^\varepsilon g_j$.
	Hence, rearranging terms gives
	\begin{equation}\label{eq: F_j + g_j symmetric}
		F_j(a) + g_j = (-1)^\varepsilon a^{p-2j} (F_{p-j}(a) + g_{p-j}) .
	\end{equation}
	Next, assume that $j = \tfrac p 2$ and $\varepsilon = 1$. 
	In this case, equation \eqref{eq: Cauchy G} is of the form 
	\[
		G_{\frac p 2}(ac) = G_{\frac p 2}(a) + G_{\frac p 2}(c) .
	\]
	This is one of Cauchy's classical functional equations. Its solution is well known to be
	\[
		G_{\frac p 2}(a) = 2 g_{\frac p 2} \, \ln(a)
	\]
	for some constant $g_{\frac p 2} \in \R$.
	In terms of $F_{\frac p 2}$ this reads as
	\begin{equation}\label{eq: F_(p/2)}
		F_{\frac p 2}(a) = g_{\frac p 2} \, \ln(a) .
	\end{equation}
	Now, we are in a position to define our desired function $\tilde F$ by
	\[
	  \tilde F_j =
		\begin{cases}
			F_j & \text{for $j \in \{ 0, \ldots, p \}$ such that $j+\varepsilon$ is even} \\
			F_j + g_j & \text{for $j \in \{ 0, \ldots, p \} \setminus \{ \frac p 2 \}$ such that $j+\varepsilon$ is odd} \\
			0 & \text{for $j = \frac p 2$ even and $\varepsilon = 1$} \\
			F_j & \text{for $j = \frac p 2$ odd and $\varepsilon = 0$.} \\
		\end{cases} 
	\]
	For $\frac p 2$ even and $\varepsilon = 1$ set $k = g_{\frac p 2}$. 
	In all other cases set $k=0$. 
	A glance at \eqref{eq: F describes mu in coordinates} reveals that an addition of constants for even $p$
	and odd $j+\varepsilon$ does not change \eqref{eq: F describes mu in coordinates}. 
	This and \eqref{eq: F_(p/2)} imply that
	\[
		\tilde F + k \ln \, \eMiddle
	\]
	describes $\mu$. We still need to show that $\tilde F$ satisfies 
	\eqref{eq: ptwotwo 0 frac 1 a a 0 tilde F(a) = tilde F(a)}. In coordinates, we have to prove
	\begin{equation}\label{eq: ptwotwo 0 frac 1 a a 0 tilde F(a) = tilde F(a) coordinates}
		\tilde F_j(a) = (-1)^{\varepsilon} a^{p-2j} \tilde F_{p-j}(a).
	\end{equation}
	Let $j = \frac p 2$. 
	If $\varepsilon = 0$, then this is trivially true.
	For $\varepsilon = 1$ we have to distinguish two cases.
	First, assume $\frac p 2$ is even.
	Then $\tilde F_{\frac p 2} = 0$ and thus 
	\eqref{eq: ptwotwo 0 frac 1 a a 0 tilde F(a) = tilde F(a) coordinates} obviously holds.
	Second, let $\frac p 2$ be odd. 
	Then \eqref{eq: F_j symmetric} implies \eqref{eq: ptwotwo 0 frac 1 a a 0 tilde F(a) = tilde F(a) coordinates}.
	For $j \neq \frac p 2$, the desired equality follows directly from \eqref{eq: F_j symmetric} 
	and \eqref{eq: F_j + g_j symmetric}. 

	Now, let $p$ be odd and $j \in \{ 0, \ldots, p \}$.
	We suppose further that $j+\varepsilon$ is even. 
	Then choosing $b = c = d = 1$ in \eqref{eq: mu[I, J] = mu[J ,I] for F} together with
	the assumption that $F_{j}(1) = 0$ yields
	\begin{equation}\label{eq: F_j symmetric p odd}
		F_j(a) + (-1)^\varepsilon F_{p-j}(1) = (-1)^\varepsilon a^{p-2j} F_{p-j}(a) .
	\end{equation}
	We are already in a position to define $\tilde F$ and $k$ by
	\[
		\tilde F_j =
		\begin{cases}
			F_j + (-1)^\varepsilon F_{p-j}(1) & \text{for $j \in \{ 0, \ldots, p \}$ such that $j+\varepsilon$ is even} \\
			F_j & \text{for $j \in \{ 0, \ldots, p \}$ such that $j+\varepsilon$ is odd}
		\end{cases}
	\]
	and $k = 0$, respectively.
	Similar to before we see that $\tilde F$ describes $\mu$. 
	Moreover, from \eqref{eq: F_j symmetric p odd} follows \eqref{eq: ptwotwo 0 frac 1 a a 0 tilde F(a) = tilde F(a) coordinates},
	which in turn yields \eqref{eq: ptwotwo 0 frac 1 a a 0 tilde F(a) = tilde F(a)}.
\end{proof}

Next, we are going to deduce a crucial linear equation.

\begin{lemma}\label{le: mu describable, linear equation, VL(2)-covariant}
	Let $\mu \in \ValCoEps{p}{2}$.
	If $\mu$ can be described by some measurable function,
	then it can also be described by a measurable $F \colon (0, \infty) \to \Sym{\R^2}$ with the following properties:
	\begin{itemize}
		\item
		There exists a tensor $C \in \Sym{\R^2}$ such that
		\begin{equation}\label{eq: F(t), VL(2)}
			F(t) =
			\ptwotwo 1 {\frac 1 {st}} 0 1 \cdot F \left( \frac {st} {s+1} \right)
			+ (-1)^\varepsilon \ptwotwo s {\frac 1 t} t 0 \cdot F \left( \frac t {s+1} \right)
			+ (-1)^\varepsilon \diag {\sqrt{st}} {\frac 1 {\sqrt{st}}} \cdot C
		\end{equation}
		holds for all $s, t > 0$.

		\item
		For all $s > 0$  we have
		\begin{equation}\label{eq: F mirrored, VL(2)}
			(-1)^\varepsilon \ptwotwo 0 {\frac 1 s} s 0 \cdot F(s) = F(s) .
		\end{equation}

		\item
		The tensor $C$ satisfies
		\begin{equation}\label{eq: tensor C, VL(2)}
			(-1)^\varepsilon \ptwotwo 0 1 1 0 \cdot C = C \quad \text{and} \quad (-1)^\varepsilon \ptwotwo {-1} {-1} 0 1 \cdot C = C .
		\end{equation}

	\end{itemize}
\end{lemma}
\begin{proof}
	By Lemma \ref{le: ptwotwo 0 frac 1 a a 0 F(a) = F(a)} there exists a function $\tilde F$ and a constant $k$ such that
	\[
		F = \tilde F + k \ln \, \eMiddle
	\]
	describes $\mu$. 
	Let us remark that, without further mentioning, we will use that $k$ vanishes in most cases.
	In fact, we know from Lemma \ref{le: ptwotwo 0 frac 1 a a 0 F(a) = F(a)} that $k = 0$ except for $p$ even, 
	$\frac p 2$ even, and $\varepsilon = 1$.
	Thus, relations such as $(-1)^{\varepsilon} k = -k$ and $(-1)^p k = k$ will be implicitly used.
	
	It easily follows from \eqref{eq: ptwotwo 0 frac 1 a a 0 tilde F(a) = tilde F(a)} that
	\begin{equation}\label{eq: F(uv)}
		(-1)^\varepsilon \ptwotwo 0 {- \frac 1 u} {-u} 0 \cdot \tilde F(uv) = (-1)^p \diag {\frac 1 v} v \cdot \tilde F(uv) 
	\end{equation}
	for all $u, v > 0$ and
	\begin{equation}\label{eq: F(t) symmetric}
		(-1)^\varepsilon \ptwotwo 0 1 t 0 \cdot \tilde F(t) = \diag 1 t \cdot \tilde F(t) 
	\end{equation}
	for all $t > 0$. 

	Let $s, t, u, v$ be positive real numbers and consider the triangle $T$ with corners 
	$s e_1 + t e_2$, $-u e_1$ and $-v e_2$.
	We can write $T$ in two different ways involving double pyramids.
	In fact, a simple calculation shows that on the one hand
	\[
		T = \left[ -u e_1, \frac {sv} {t+v} e_1, -v e_2, t \ptwo {\frac s t} 1 \right],
	\]
	and on the other hand
	\[
		T = \ptwotwo 0 1 {-1} 0 \left[ -v e_1, \frac {tu} {s+u} e_1, -s \ptwo {- \frac t s} 1, u e_2 \right] .
	\]
	Since $F$ describes $\mu$, equation \eqref{eq: mu describable by a function F} holds. 
	Applying this to the first representation of $T$ gives
	\begin{align*}
		\mu (T) = &
		(-1)^p \diag {\frac 1 v} v \cdot F(uv) +
		(-1)^\varepsilon \diag {\frac 1 v} {-v} \cdot F \left( \frac {sv^2} {t+v} \right) \\
		& {} + (-1)^\varepsilon \ptwotwo {- \frac 1 t} 0 s t \cdot F(tu) + 
		\ptwotwo {\frac 1 t} 0 s t \cdot F \left( \frac {stv} {t+v} \right),
	\end{align*}
	whereas the second representation and the $\SL[2]$-covariance of $\mu$ yield
	\begin{align*}
		\mu (T) = &
		(-1)^p \ptwotwo 0 {\frac 1 s} {-s} {-t} \cdot F(sv) +
		(-1)^\varepsilon \ptwotwo 0 {\frac 1 s} s t \cdot F \left( \frac {stu} {s+u} \right) \\
		& {} + (-1)^\varepsilon \ptwotwo 0 {- \frac 1 u} {-u} 0 \cdot F(uv) +
		\ptwotwo 0 {\frac 1 u} {-u} 0 \cdot F \left( \frac {tu^2} {s+u} \right) .
	\end{align*}
	So the right hand sides of the last two equations must be equal. 
	In the resulting equation we can plug in the definition of $F$ and use \eqref{eq: F(uv)}.
	By doing so, it turns out that the terms containing $u$ and those containing $v$ can be separated. 
  We therefore arrive at
	\begin{equation}\label{eq: L = R}
		L(s, t, u) = R(s, t, v) ,
	\end{equation}
	where
	\begin{align*}
		L(s, t, u) := &
		(-1)^\varepsilon \ptwotwo {- \frac 1 t} 0 s t \cdot F(tu) -
		(-1)^\varepsilon \ptwotwo 0 {\frac 1 s} s t \cdot F \left( \frac {stu} {s+u} \right) \\
		& {} - \ptwotwo 0 {\frac 1 u} {-u} 0 \cdot F \left( \frac {tu^2} {s+u} \right) +
		2k \ln(u) \, \eMiddle 
	\end{align*}
	and
	\begin{align*}
		R(s, t, v) := &
		\ptwotwo 0 {- \frac 1 s} s t \cdot F(sv) -
		\ptwotwo {\frac 1 t} 0 s t \cdot F \left( \frac {stv} {t+v} \right) \\
		& {} - (-1)^\varepsilon \diag {\frac 1 v} {-v} \cdot F \left( \frac {sv^2} {t+v} \right) -
		2k \ln(v) \, \eMiddle .
	\end{align*}
  A straight forward calculation proves
	\begin{equation}\label{eq: L(s, t, s)}
		\diag {\frac 1 s} s \cdot L(s, t, s) = L(1, st, 1) + 2k \ln(s) \,\eMiddle.
	\end{equation}
	Similarly, we have
	\begin{equation}\label{eq: R(s, t, t)}
		(-1)^\varepsilon \ptwotwo 0 t {\frac 1 t} 0 \cdot R(s, t, t) = L(1, st, 1) + 2k \ln(t) \, \eMiddle .
	\end{equation}
	Choose $s = t = 1$ in \eqref{eq: R(s, t, t)}. A glance at \eqref{eq: L = R} then shows
	\begin{equation}\label{eq: L(1,1,1) R(1,1,1)}
		L(1,1,1) = (-1)^\varepsilon \ptwotwo 0 1 1 0 \cdot R(1,1,1)
		         = (-1)^\varepsilon \ptwotwo 0 1 1 0 \cdot L(1,1,1) .
	\end{equation}
	From \eqref{eq: L = R} we infer that $L(s, t, u)$ is independent of $u$.
	In particular, $L(s, t, u) = L(s, t, s)$.
	By \eqref{eq: L(s, t, s)} we therefore have
	\[
		L(s, t, u) = \diag s {\frac 1 s} \cdot L(1, st, 1) + 2k \ln(s) \, \eMiddle .
	\]
	Consequently,
	\[
		\ptwotwo {-t} 0 s {\frac 1 t} \cdot L(s, t, u) =
		\ptwotwo {-st} 0 1 {\frac 1 {st}} \cdot L(1, st, 1) + 2k \ln(s) \ptwotwo {-1} 0 {st} 1 \cdot \eMiddle .
	\]
	Set $u = 1$ in this equation and plug in the definition of $L(s, t, 1)$ afterwards.
	Then we obtain
	\begin{multline*}
		F(t) -
		\ptwotwo 1 {\frac 1 {st}} 0 1 \cdot F \left( \frac {st} {s+1} \right) -
		(-1)^\varepsilon \ptwotwo s {\frac 1 t} t 0 \cdot F \left( \frac t {s+1} \right) = \\
		(-1)^\varepsilon \ptwotwo {-st} 0 1 {\frac 1 {st}} \cdot L(1, st, 1) -
		2k \ln(s) \ptwotwo {-1} 0 {st} 1 \cdot \eMiddle .
	\end{multline*}
	Define a function $\tilde H \colon (0, \infty) \to \Sym{\R^2}$ by
	\[
		\tilde H(s) = \ptwotwo {-s} 0 1 {\frac 1 s} \cdot L(1, s, 1).
	\]
	Using this definition, the last equation reads as
	\begin{multline}\label{eq: equation for F with tilde H}
		F(t) -
		\ptwotwo 1 {\frac 1 {st}} 0 1 \cdot F \left( \frac {st} {s+1} \right) -
		(-1)^\varepsilon \ptwotwo s {\frac 1 t} t 0 \cdot F \left( \frac t {s+1} \right) = \\
		(-1)^\varepsilon \tilde H(st)  -
		2k \ln(s) \ptwotwo {-1} 0 {st} 1 \cdot \eMiddle .
	\end{multline}
	This is a functional equation for $F$ whose homogeneous part coincides with the one 
	of the desired equation \eqref{eq: F(t), VL(2)}.
	However, we still need to simplify the inhomogeneous part.
	A multiplication of \eqref{eq: equation for F with tilde H} by
	\[
		\diag 1 t
	\]
  proves
  \begin{multline}\label{eq: diag 1 t tilde tilde H(st), first}
		\diag 1 t \cdot F(t) -
		\ptwotwo 1 {\frac 1 s} 0 t \cdot F \left( \frac {st} {s+1} \right) -
		(-1)^\varepsilon \ptwotwo s 1 t 0 \cdot F \left( \frac t {s+1} \right) = \\
		(-1)^\varepsilon \diag 1 t \cdot \tilde H(st) -
		2k \ln(s) \ptwotwo {-1} 0 {st} t \cdot \eMiddle .
	\end{multline}
	Next, we replace $s$ by $\frac 1 s$ in \eqref{eq: diag 1 t tilde tilde H(st), first} and multiply the equation by
	\[
		(-1)^\varepsilon \ptwotwo 0 1 1 0 .
	\]
	This yields the following:
	\begin{multline}\label{eq: diag 1 t tilde tilde H(st), second}
		(-1)^\varepsilon \ptwotwo 0 1 t 0 \cdot F(t) -
		(-1)^\varepsilon \ptwotwo s 1 t 0 \cdot F \left( \frac t {s+1} \right) -
		\ptwotwo 1 {\frac 1 s} 0 t \cdot F \left( \frac {st} {s+1} \right) = \\
		\ptwotwo 0 1 t 0 \cdot \tilde H \left( \frac t s \right) -
		2k \ln(s) \ptwotwo 0 {-1} t {\frac t s} \cdot \eMiddle .
	\end{multline}
	We subtract \eqref{eq: diag 1 t tilde tilde H(st), first} from \eqref{eq: diag 1 t tilde tilde H(st), second}
  and see that on the resulting left hand side only terms involving $F(t)$ remain.
	For these terms we plug in the definition of $F$ and apply \eqref{eq: F(t) symmetric}.
	We then arrive at 
	\begin{multline*}
			(-1)^\varepsilon \diag 1 t \cdot \tilde H(st)
			- 2k \ln(s) \ptwotwo {-1} 0 {st} t \cdot \eMiddle
      - k \ln(t) \diag 1 t \cdot \eMiddle = \\
			\ptwotwo 0 1 t 0 \cdot \tilde H \left( \frac t s \right)
			- 2k \ln(s) \ptwotwo 0 {-1} t {\frac t s} \cdot \eMiddle
      + k \ln(t) \ptwotwo 0 1 t 0 \cdot \eMiddle .
	\end{multline*}
	Setting $s = t$ and rearranging terms yields
	\begin{multline}\label{eq: tilde H (t^2) = 1}
		\tilde H \left( t^2 \right) =
		(-1)^\varepsilon \ptwotwo 0 {\frac 1 t} t 0 \cdot \tilde H(1) +
		2k \ln(t) \ptwotwo 0 {-1} 1 {\frac 1 {t^2}} \cdot \eMiddle \\
		{} - 2k \ln(t) \eMiddle -
		2k \ln(t) \ptwotwo {-1} 0 {t^2} 1 \cdot \eMiddle .
	\end{multline}
	Now, choose $t = 1$. Then we obtain
	\begin{equation}\label{eq: tilde H(1) symmetric}
		\tilde H(1) = (-1)^\varepsilon \ptwotwo 0 1 1 0 \cdot \tilde H(1) .
	\end{equation}
	If we plug this back into \eqref{eq: tilde H (t^2) = 1}, then we obviously get
	\begin{multline}\label{eq: H(t^2)}
		\tilde H \left( t^2 \right) =
		\diag t {\frac 1 t} \cdot \tilde H(1) +
		2k \ln(t) \ptwotwo 0 {-1} 1 {\frac 1 {t^2}} \cdot \eMiddle \\
		{} - 2k \ln(t) \eMiddle -
		2k \ln(t) \ptwotwo {-1} 0 {t^2} 1 \cdot \eMiddle .
	\end{multline}
	Moreover, by the definition of $\tilde H$, relation \eqref{eq: L(1,1,1) R(1,1,1)}, 
	and the definition of $\tilde H$ again we have
	\begin{align*}
		\tilde H(1) &= \ptwotwo {-1} 0 1 1 \cdot L(1,1,1) \\
		&= (-1)^\varepsilon \ptwotwo {-1} 0 1 1 \ptwotwo 0 1 1 0 \cdot L(1,1,1) \\
		&= (-1)^\varepsilon \ptwotwo {-1} 0 1 1 \ptwotwo 0 1 1 0 {\ptwotwo {-1} 0 1 1}^{-1} \cdot \tilde H(1) \\
		&= (-1)^\varepsilon \ptwotwo {-1} {-1} 0 1 \cdot \tilde H(1) .
	\end{align*}
  Define $C = \tilde H(1)$. 
  Then \eqref{eq: tilde H(1) symmetric} and the last lines prove \eqref{eq: tensor C, VL(2)}.
	Furthermore, set
	\begin{multline*}
		H(s, t) =
		(-1)^\varepsilon \diag {\sqrt{st}} {\frac 1 {\sqrt{st}}} \cdot C \\
		{} - k \ln(st) \ptwotwo 0 {-1} 1 {\frac 1 {st}} \cdot \eMiddle +
		k \ln(st) \, \eMiddle \\
		{} + k \ln(st) \ptwotwo {-1} 0 {st} 1 \cdot \eMiddle -
		2k \ln(s) \ptwotwo {-1} 0 {st} 1 \cdot \eMiddle .
	\end{multline*}
	Now, replace $t$ by $\sqrt{st}$ in \eqref{eq: H(t^2)}.
	This yields a formula for $\tilde H(st)$. 
	Plugging this formula into \eqref{eq: equation for F with tilde H} shows that $F$ satisfies
	\begin{equation}\label{eq: functional F = H(s,t)}
		F(t) =
		\ptwotwo 1 {\frac 1 {st}} 0 1 \cdot F \left( \frac {st} {s+1} \right)
		+ (-1)^\varepsilon \ptwotwo s {\frac 1 t} t 0 \cdot F \left( \frac t {s+1} \right)
		+ H(s, t) .
	\end{equation}
  If we can show that $k = 0$, then we are done.
  Indeed, if $k$ vanishes, then the definition of $H(s, t)$ and the last equation prove
  the desired functional equation \eqref{eq: F(t), VL(2)} for $F$. 
  Moreover, the definition of $F$ and \eqref{eq: ptwotwo 0 frac 1 a a 0 tilde F(a) = tilde F(a)}
  would imply \eqref{eq: F mirrored, VL(2)}.
  
	So let us turn to the proof that $k = 0$.
	Let $p$ be even such that $\frac p 2$ is also even and suppose that $\varepsilon = 1$.
	Recall that this is the only combination of $p$ and $\varepsilon$ for which
	we have to prove something since in all other cases we already know 
	from Lemma \ref{le: ptwotwo 0 frac 1 a a 0 F(a) = F(a)} that $k = 0$.
	
	First, assume $p = 0$. If we plug $F = k \ln$ into \eqref{eq: mu describable by a function F},
	then for all $k$ the right hand side of this equation is always equal to $0$.
	Hence, we can just set $k = 0$.
	
	Second, suppose that $p \neq 0$.
	Multiplying \eqref{eq: functional F = H(s,t)} by 
	\[
		\ptwotwo 1 {\frac 1 t} 0 1
	\]
	and then setting $s = \frac x y$ and $t = x+y$ yields
	\[
		\ptwotwo 1 {\frac 1 {x+y}} 0 1 \cdot F(x+y) =
		\ptwotwo 1 {\frac 1 x} 0 1 \cdot F(x) -
		\ptwotwo {\frac x y} {\frac 1 y} {x+y} 1 \cdot F(y) +
		\ptwotwo 1 {\frac 1 {x+y}} 0 1 \cdot H \left( \frac x y , x+y \right)
	\]
	for all $x, y > 0$.
	Define a function $G \colon (0,\infty) \to \Sym{\R^2}$ by
	\[
		G(x) = \ptwotwo 1 {\frac 1 x} 0 1 \cdot F(x) .
	\]
	Then the previous equation becomes
	\[
		G(x+y) = G(x) - \ptwotwo {-1} 0 {x+y} 1 \cdot G(y) + \ptwotwo 1 {\frac 1 {x+y}} 0 1 \cdot H \left( \frac x y , x+y \right) .
	\]
	From this it follows easily that the $p$-th component of $G$ satisfies
	\begin{equation}\label{eq: G_p function equation}
		G_p(x+y) - G_p(x) + G_p(y) = \left[ \ptwotwo 1 {\frac 1 {x+y}} 0 1 \cdot H \left( \frac x y , x+y \right) \right]_p .
	\end{equation}
	In particular, by setting $x = y$, we derive that 
	\begin{equation}\label{eq: G_p(x) = H(1,x)}
		G_p(x) = \left[ \ptwotwo 1 {\frac 1 x} 0 1 \cdot H(1, x) \right]_p .
	\end{equation}
	Next, we determine the inhomogeneity of \eqref{eq: G_p function equation} explicitely.
	By the definition of $H$ we have
	\begin{align*}
		\ptwotwo 1 {\frac 1 {x+y}} 0 1 \cdot H \left( \frac x y , x+y \right) = \, &
		(-1)^\varepsilon \ptwotwo 1 {\frac 1 {x+y}} 0 1 \diag {\sqrt{\frac{x(x+y)}{y}}} {\sqrt{\frac{y}{x(x+y)}}} \cdot C \\
		& {} - k \ln \left( \frac {x(x+y)} {y} \right) \ptwotwo 0 {-1} 1 {\frac 1 x} \cdot \eMiddle \\
		& {} + k \ln \left( \frac {x(x+y)} {y} \right) \ptwotwo 1 {\frac 1 {x+y}} 0 1 \cdot \eMiddle \\
		& {} + k \ln \left( \frac {x(x+y)} {y} \right) \ptwotwo {-1} {-\frac {1} {x+y}} {\frac {x(x+y)} {y}} {\frac {x+y} {y}} \cdot \eMiddle \\
		& {} - 2k \ln \left( \frac x y \right) \ptwotwo {-1} {-\frac {1} {x+y}} {\frac {x(x+y)} {y}} {\frac {x+y} {y}} \cdot \eMiddle.
	\end{align*}	
	Recall that $\frac p 2$ is even and $\varepsilon = 1$.
	Thus, applying \eqref{eq: tilt expanded symmetric transpose} to the first term 
	and \eqref{eq: pth component middle term} to the other ones gives
	\begin{multline*}
		\left[ \ptwotwo 1 {\frac 1 {x+y}} 0 1 H \left( \frac x y , x+y \right) \right]_p =
		- (x+y)^{- \frac p 2} \sum_{j=0}^p C_j \left( \frac x y \right)^{\frac p 2 - j}
		\\ {} + k \left( (x+y)^{- \frac p 2} - x^{- \frac p 2} + y^{- \frac p 2} \right) \ln \left( \frac {x(x+y)} y \right)
		- 2k y^{- \frac p 2} \ln \left( \frac x y \right) .
	\end{multline*}
	If we set $x = y$ in this equation and recall relation \eqref{eq: G_p(x) = H(1,x)}, then we arrive at
	\[
		G_p(x) = x^{- \frac p 2} \left( k \ln(x) - \sum_{j = 0}^p C_j \right) .
	\]
	Plugging the last two expressions into \eqref{eq: G_p function equation} yields
	\begin{multline*}
		(x+y)^{- \frac p 2} \left( k \ln(x+y) - \sum_{j = 0}^p C_j \right)
		- x^{- \frac p 2} \left( k \ln(x) - \sum_{j = 0}^p C_j \right)
		+ y^{- \frac p 2} \left( k \ln(y) - \sum_{j = 0}^p C_j \right)
		= \\
		- (x+y)^{- \frac p 2} \sum_{j=0}^p C_j \left( \frac x y \right)^{\frac p 2 - j}
		+ k \left( (x+y)^{- \frac p 2} - x^{- \frac p 2} + y^{- \frac p 2} \right) \ln \left( \frac {x(x+y)} y \right)
		- 2k y^{- \frac p 2} \ln \left( \frac x y \right) ,
	\end{multline*}
	which we rewrite to
	\begin{multline*}
		k \ln(x) = \left( (x+y)^{- \frac p 2} - y^{- \frac p 2} \right)^{-1} 
		\left[ \left(x^{- \frac p 2} - y^{- \frac p 2} - (x+y)^{- \frac p 2} \right) \sum_{j = 0}^p C_j \right.
		\\ \left. + {} (x+y)^{- \frac p 2} \sum_{j=0}^p C_j \left( \frac x y \right)^{\frac p 2 - j} + 
		  k \left( (x+y)^{- \frac p 2} - x^{- \frac p 2} \right) \ln(y)
		+ k \left( x^{- \frac p 2} - y^{- \frac p 2} \right) \ln(x+y) \right].
	\end{multline*}
	Note that in this step we used the fact that $p \neq 0$.
	Fix a $y > 0$. 
	Using the standard branch of the logarithm, the left hand side can be extended to a holomorphic function on
	$\C \setminus (- \infty, 0]$ whereas the right hand side can be extended to
	a meromorphic function on $\C \setminus (- \infty, -y]$.
	If $k \neq 0$, then the identity theorem for holomorphic functions would imply that the left hand side
	could be further extended continuously at some point on the negative real axis, which is impossible.
	Thus, $k$ has to be zero. 
\end{proof}

The last lemma shows that there exists a describing function $F$ which satisfies a linear functional equation.
By solving this equation, we will now describe such functions completely.

\begin{lemma}\label{le: solve equation, VL(2)}
	Let $\mu \in \ValCoEps{p}{2}$ be described by some measurable function.
	For $p = 0$ and $\varepsilon = 0$ there exist constants $k_1, k_2 \in \R$ such that
	\[
		F(x) = k_1 x + k_2 , \quad x > 0 ,
	\]
	describes $\mu$.
	For $p = 0$ and $\varepsilon = 1$ the function $F = 0$ describes $\mu$.
	
	Let $p \geq 1$. Unless $p$ is odd and $\varepsilon = 1$,
	there exists a tensor $K \in \Sym{\R^2}$ with
	\[
		K_{p-1} = 0 \qquad \text{and} \qquad (-1)^\varepsilon \diag {-1} 1 \cdot K = K 
	\]
	such that
	\begin{equation}\label{eq: F sol K}
		F(x) = \int_0^x \ptwotwo 1 {- \frac 1 x} z {1 - \frac z x} \cdot K \ dz , \quad x > 0 ,
	\end{equation}
	describes $\mu$.

\end{lemma}
\begin{proof}
	We can assume that $\mu$ is described by a function $F \colon (0, \infty) \to \Sym{\R^2}$ that satisfies
	the conclusions of Lemma \ref{le: mu describable, linear equation, VL(2)-covariant}.
	As in the proof of Lemma \ref{le: solve homogeneous equation},
	multiplying \eqref{eq: F(t), VL(2)} by 
	\[
		\ptwotwo 1 {\frac 1 t} 0 1
	\]
	and then setting $s = \frac x y$ and $t = x+y$ yields	
	\begin{align*}
		\ptwotwo 1 {\frac 1 {x+y}} 0 1 \cdot F(x+y) = \, &
		\ptwotwo 1 {\frac 1 x} 0 1 \cdot F(x) +
		(-1)^\varepsilon \ptwotwo {\frac x y} {\frac 1 y} {x+y} 1 \cdot F(y) \\
		& {} + (-1)^\varepsilon \ptwotwo {\sqrt {\frac {x (x+y)} {y}}} {\sqrt {\frac {x} {y (x+y)}}} 0 {\sqrt {\frac {y} {x (x+y)}}} \cdot C
	\end{align*}
	for all $x, y > 0$.
	Define a function $G \colon (0,\infty) \to \Sym{\R^2}$ by
	\[
		G(x) = \ptwotwo 1 {\frac 1 x} 0 1 \cdot F(x) .
	\]
	By the symmetry relation \eqref{eq: F mirrored, VL(2)} we have
	\[
		(-1)^\varepsilon \ptwotwo {-1} 0 x 1 \cdot G(x) = G(x) .
	\]
	Thus, an elementary calculation yields 
	\[
		G(x+y) = G(x) + \ptwotwo 1 0 x 1 \cdot G(y) +
		(-1)^\varepsilon \ptwotwo {\sqrt {\frac {x (x+y)} {y}}} {\sqrt {\frac {x} {y (x+y)}}} 0 {\sqrt {\frac {y} {x (x+y)}}} \cdot C .
	\]
  Using \eqref{eq: tilt expanded symmetric transpose}, 
  it is not hard to determine the $p$-th component of the last term in this equation.
  Thus, the $p$-th component of $G$ satisfies
	\begin{equation}\label{eq: G_p(x+y)}
		G_p(x+y) = G_p(x) + G_p(y) + (-1)^\varepsilon (x+y)^{- \frac p 2} \sum_{j = 0}^p C_j \left( \frac x y \right)^{\frac p 2 - j} .
	\end{equation}  
	
	First, let $p = 0$. 
	For $\varepsilon = 0$, equation \eqref{eq: G_p(x+y)} simplifies to
	\[
		G_0(x+y) = G_0(x) + G_0(y) + C_0.
	\]
	This is an inhomogeneous version of Cauchy's functional equation.
	Therefore, the solution is given by $G_0(x) = k_1 x + k_2$ for some constant $k_1$ and $k_2 = -C_0$.
	Since $F = G_0$ for $p = 0$, the case $\varepsilon = 0$ is settled.
	If $\varepsilon = 1$, then equation \eqref{eq: F mirrored, VL(2)} directly implies $F = 0$.
	This concludes the proof of the scalar case $p = 0$.
	
	Second, let $p \geq 1$.
	For $x, y > 0$ define 
	\begin{equation}\label{eq: def h}
		h(x, y) = (-1)^\varepsilon (x+y)^{- \frac p 2} \sum_{j = 0}^p C_j \left( \frac x y \right)^{\frac p 2 - j} .
	\end{equation}
	With this definition, equation \eqref{eq: G_p(x+y)} simplifies to
	\[
	 G_p(x+y) = G_p(x) + G_p(y) + h(x, y).
	\]
	This clearly implies that $h$ is a symmetric function.
	By the last relation, we can calculate $G_p(x+y+1)$ in two different ways.
	On the one hand,
	\begin{align*}
		G_p(x+y+1)
		&= G_p(x) + G_p(y+1) + h(x, y+1) \\
		&= G_p(x) + G_p(y) + G_p(1) + h(x, y+1) + h(1, y)
	\end{align*}
	and on the other hand
	\begin{align*}
		G_p(x+y+1)
		&= G_p(x+1) + G_p(y) + h(x+1, y) \\
		&= G_p(x) + G_p(y) + G_p(1) + h(x+1, y) + h(x, 1) .
	\end{align*}
	Consequently, for all $x, y > 0$ we have
	\begin{equation}\label{eq: h(x, y+1)}
		h(x, y+1) + h(1, y) = h(x+1, y) + h(x, 1) .
	\end{equation} 
	
	Definition \eqref{eq: def h} clearly extends to all $x \in \C$ with $0 < |x| < |y|$.
	Let $y > 2$ and write $B_1^\times(0)$ for the punctured unit disc $\{ x \in \C \colon 0 < |x| < 1\}$. 
	The identity theorem for holomorphic functions and \eqref{eq: h(x, y+1)} imply that
	\begin{equation}\label{eq: h(x^2, y+1)}
		h(x^2, y+1) + h(1, y) = h(x^2+1, y) + h(x^2, 1) 
	\end{equation}
	holds for all $x \in B_1^\times(0)$.
	Therefore, we can use this equation to compare coefficients of the respective Laurent series.
	We consider the Laurent expansions at zero 
	and write, for example, $\left[ x^j \right] h(x^2,y)$ for the coefficient of $x^j$ in the Laurent expansion of $x \mapsto h(x^2, y)$.
	The functions $x \mapsto h(1, y)$ and $x \mapsto h(x^2+1, y)$ are holomorphic on $B_1(0)$. Hence,
	\[
		\left[ x^j \right] h(1, y) = \left[ x^j \right] h(x^2+1, y) = 0
	\]
	for $j < 0$.
	Moreover, we obviously have
	\[
		\left[ x^0 \right] h(1, y) = \left[ x^0 \right] h(x^2+1, y).
	\]
	So by \eqref{eq: h(x^2, y+1)} we deduce for all $j \leq 0$ that
	\begin{equation}\label{eq: [x^j]}
		\left[ x^j \right] h(x^2, y+1) = \left[ x^j \right] h(x^2, 1) .
	\end{equation}
	We need a series expansion of $x \mapsto h(x,y)$.
	Since $h$ is symmetric, we can also look at $x \mapsto h(y,x)$.
	For this map, using the Taylor expansion of $x \mapsto (x + y)^{-\frac p 2}$ at zero
	and the first relation of \eqref{eq: tensor C, VL(2)},
  we obtain by a rearrangement of the involved sums that
	\[
		h(x, y) = \sum_{i = 0}^\infty \sum_{j = 0}^{i \wedge p }\binom {-\frac p 2} {i-j} C_{p-j} x^{i-\frac p 2} y^{-i}
	\]	
	for all $x \in \R$ with $0 < x < y$. 
	Here, $i \wedge p$ denotes the minimum of $i$ and $p$.	
	This relation directly yields the Laurent expansion of $h(x^2, y)$ at zero. 
	Thus, \eqref{eq: [x^j]} for the coefficients of $x^{2i-p}$ gives
	\[
		\sum_{j = 0}^i \binom {- \frac p 2} {i-j} C_{p-j} (y+1)^{-i} =
		\sum_{j = 0}^i \binom {- \frac p 2} {i-j} C_{p-j}
	\]
	for $i \in \{ 0, \ldots, \left\lfloor \frac p 2 \right\rfloor \}$.
	Clearly, this can only hold if
	\begin{equation}\label{eq: sum binom C}
		\sum_{j = 0}^i \binom {- \frac p 2} {i-j} C_{p-j} = 0 
	\end{equation}
	for $i \in \{ 1, \ldots, \left\lfloor \frac p 2 \right\rfloor \}$.
	
	We will now prove that $C_p$ determines all other components of $C$.
	This is clear once we have shown that
	\begin{equation}\label{eq: C_l}
		(-1)^\varepsilon C_i = C_{p-i} = \binom {\frac p 2} i C_p
	\end{equation}
	for $i \in \{ 0, \ldots, \left\lfloor \frac p 2 \right\rfloor \}$.
	The first relation of \eqref{eq: tensor C, VL(2)} in coordinates is precisely the first equation.
	The second is shown by induction on $i$.
	Clearly, the desired equality holds for $i = 0$.
	So let $i \geq 1$ and assume that it holds for all $j < i$.
	Then the induction assumption and the Vandermonde identity \eqref{eq: Vandermonde} yield
	\begin{eqnarray*}
		\sum_{j = 0}^i \binom {- \frac p 2} {i-j} C_{p-j} &=& 
		C_p \sum_{j = 0}^{i-1} \binom {- \frac p 2} {i-j} \binom {\frac p 2} j + C_{p-i} \\
		&=& - \binom {\frac p 2} i C_p + C_{p-i}.
	\end{eqnarray*}
	Now \eqref{eq: sum binom C} implies \eqref{eq: C_l}.

	First, suppose that $p+\varepsilon$ is odd.
	Looking at the $0$-th coordinate of the second part of \eqref{eq: tensor C, VL(2)}, we see that
	\[
		C_0 = (-1)^{p+\varepsilon} C_0 .
	\]
	So $C_0 = 0$ and, by \eqref{eq: C_l}, also $C = 0$.
	From \eqref{eq: F(t), VL(2)}, \eqref{eq: F mirrored, VL(2)}, and Lemma \ref{le: solve homogeneous equation} 
	we therefore conclude that $F$ has the desired form.
	
	Second, let $p$ as well as $\frac p 2$ be even and suppose that $\varepsilon = 0$.
	From the second part of \eqref{eq: tensor C, VL(2)}, equation \eqref{eq: tilt expanded symmetric transpose},
	an index change, relation \eqref{eq: C_l}, and the binomial theorem we get  
	\begin{align*}
		C_p &=\left[ \ptwotwo {-1} {-1} 0 1 C \right]_p 
		=	\sum_{j = 0}^p (-1)^{p-j} C_j 
		= \sum_{j = 0}^{\frac p 2} (-1)^j C_j + \sum_{j = 0}^{\frac p 2 - 1} (-1)^j C_{p-j} \\
		&= 2C_p\sum_{j = 0}^{\frac p 2} \binom {\frac p 2} j (-1)^j - (-1)^{\frac p 2} C_p  
    = - C_p.
	\end{align*}
	Thus, $C_p$ is equal to zero.
	As before, this implies $C = 0$ and that $F$ has the desired form.
	
	Third, assume that $p$ is even, $\frac p 2$ is odd and $\varepsilon = 0$.
	Using \eqref{eq: tilt expanded symmetric transpose} and \eqref{eq: C_l}, 
	it is not hard to see that the constant function 
	\[
		- C_p \, \eMiddle
	\]
	is a solution for \eqref{eq: F(t), VL(2)} and \eqref{eq: F mirrored, VL(2)}.
	Thus, the sum
	\[
		\tilde F(x) = F(x) + C_p \, \eMiddle
	\]
	satisfies \eqref{eq: F(t)} and \eqref{eq: F mirrored}. 
	Lemma \ref{le: solve homogeneous equation} again shows that $\tilde F$ has the desired form.
	Moreover, a glance at \eqref{eq: F describes mu in coordinates} reveals that, for this combination 
	of $p$, $\frac p 2$, and $\varepsilon$, 
	we can add a constant term to the $\frac p 2$-th coordinate without changing the fact that $F$ describes $\mu$.
	Thus, $\tilde F$ also describes $\mu$.
	
  Finally, we have to show that we can choose $K_{p-1} = 0$.
	For $\varepsilon = 0$ we already know $K_{p-1} = 0$ from Lemma \ref{le: solve homogeneous equation}.
	So let $p$ be even and $\varepsilon = 1$ and set $K = e_1 \odot e_2^{\odot p-1}$.
	If we plug this into \eqref{eq: F sol K}, then we get
	\[
		F(x) = \diag x 1 \cdot \int_0^x \ptwotwo {\frac 1 x} {-\frac 1 x} {\frac z x} {1-\frac z x} \cdot e_1 \odot e_2^{\odot p-1} \ dz.
	\]
	The substitution $z/x = u$ and the definition of the action $\cdot$ show
	\[
		F(x) = \diag x 1 \cdot \int_0^1 \ptwo 1 {-1} \odot \ptwo u {1-u}^{\odot p-1} \ du .
	\]
	An application of Lemma \ref{le: M^{j,p-j} help} with $b = 1$, $c = -1$ and $i = p-1$ proves
	\[
		F(x) = \frac {x^p} p e_1^{\odot p} - \frac 1 p e_2^{\odot p} .
	\]
	Plugging this into \eqref{eq: mu describable by a function F},
	we see that the right hand side is always equal to $0$.
	This completes the proof.
\end{proof}

Finally, our main result in the planar case can be deduced by a counting argument.

\begin{proof}[Proof of Theorem \ref{th: 2-dim, SL(2)-covariant}]
	Suppose that $p \geq 1$.
	We denote by $\ValCoDescr{p}{2}{\varepsilon}$ the vector space of valuations 
	from $\ValCoEps{p}{2}$ which can be described by some function $F$.
	Furthermore, define two subspaces of $\Sym{\R^2}$ by
  \[
  	V_{\varepsilon} = \left \{ K \in \Sym{\R^2} \colon 	K_{p-1} = 0 \text{ and } (-1)^\varepsilon \diag {-1} 1 \cdot K = K \right \}.
  \]
	Unless $p$ is odd and $\varepsilon = 1$, Lemma \ref{le: solve equation, VL(2)}, the $\SL[2]$-$\varepsilon$-covariance of $\mu$,
	and Theorem \ref{th: determined by values on SL(n)(R^n)} show the existence of an injective linear map
	from $\ValCoDescr{p}{2}{\varepsilon}$ to $V_\varepsilon$.
	An immediate consequence is that the inequality
	\begin{equation}\label{eq: dimensions}
		\dim \ValCoDescr{p}{2}{\varepsilon} \leq \dim V_{\varepsilon}
	\end{equation}
	holds unless $p$ is odd and $\varepsilon = 1$.

	First, let $p$ be even.
	By Lemma \ref{le: mu of double pyramid, VL(2)-covariant}, the $\SL[2]$-$\varepsilon$-covariance of $\mu$,
	and Theorem \ref{th: determined by values on SL(n)(R^n)} we have
	\[
		\dim \ValCoZero{p}{2} \leq \dim \ValCoDescr{p}{2}{0} + 1
	\]
	and
	\[
		\dim \ValCoOne{p}{2} = \dim \ValCoDescr{p}{2}{1} .
	\]
	Using \eqref{eq: Val direct sum}, the two relations above, and \eqref{eq: dimensions}, we conclude
	\begin{equation}\label{eq: dim formula}
		\dim \ValCo{p}{2} \leq \dim V_0 + \dim V_1 + 1 .
	\end{equation}
	In coordinates, the second condition on $K$ in the definition of $V_{\varepsilon}$ reads as 
	$(-1)^{\varepsilon + p + i}K_i = K_i$, $i \in \{ 0, \ldots, p \}$. 
	Therefore, the dimensions of $V_{\varepsilon}$ satisfy
	\[
		\dim V_0 = \frac p 2 + 1\qquad \text{and} \qquad \dim V_1 = \frac p 2 - 1 .
	\]
	Now, \eqref{eq: dim formula} implies that
	$\ValCo{p}{2}$ is at most $(p+1)$-dimensional.
	
	Second, let $p$ be odd. 
	Assume that $\varepsilon = 0$.
	In this case we have 
	\[
		\dim V_0 = \frac {p-1}{2} + 1.
	\]
	From \eqref{eq: dimensions} we deduce that
	\[
	  \dim \ValCoDescr{p}{2}{0} \leq \frac {p-1}{2} + 1.
	\]
	Note that for this combination of $p$ and $\varepsilon $, the constant $k$ in Lemma \ref{le: solve equation, VL(2)} vanishes. 
	Hence, $\ValCoZero{p}{2} = \ValCoDescr{p}{2}{0}$, which in turn gives
	\begin{equation}\label{eq: dim p odd}
		\dim \ValCoZero{p}{2} \leq \frac {p-1}{2} + 1.
	\end{equation}
	Consider the map $\mathcal R \colon \ValCoZero{p}{2} \to \ValCoOne{p}{2}$ defined by
	\[
		\mathcal R(\mu)(P) = \rho \cdot \mu (P^*) , \quad P \in \cP^2_o ,
	\]
	where as before, $\rho$ denotes the counter-clockwise rotation about an angle of $\frac \pi 2$.
	Since	$\mathcal R \circ \mathcal R = - \Id$, the map $\mathcal R$ is an isomorphism.
  Consequently, the spaces $\ValCoZero{p}{2}$ and $\ValCoOne{p}{2}$ have the same dimension.
  From \eqref{eq: dim p odd} and \eqref{eq: Val direct sum} we infer that also in this case
  $\ValCo{p}{2}$ is at most $(p+1)$-dimensional.
	
	Since the $p+1$ valuations from the statement of the theorem are linearly independent and have the desired properties,
	the proof is completed.
	
	For $p = 0$ we can argue analogously.
	We get
	\[
		\dim \ValCoZero{0}{2} = 3 \quad \text{and} \quad \dim \ValCoOne{0}{2} = 0 .
	\]
\end{proof}

We conclude this section with the dual result of Theorem \ref{th: 2-dim, SL(2)-covariant}.
A map $\mu \colon \cP_o^n \to \Sym{\R^n}$ is said to be $\SL$-contravariant if
\[
	\mu(\phi P) = \phi^{-t} \cdot \mu(P)
\]
for all $P \in \cP_o^n$ and each $\phi \in \SL$.
The vector space of all measurable $\SL$-contravariant valuations will be denoted by
$\ValCon{p}{n}$.

\begin{theorem}\label{th: 2-dim, SL(2)-contravariant}
	For $n = 2$ the following holds.
	\begin{itemize}
		\item
		A basis of $\ValCon{0}{2}$ is given by $\chi$, $V$ and $V \circ \, ^*$.

		\item
    For $p \geq 1$, a basis of $\ValCon{p}{2}$ is given by $\rho \cdot M^{i, p-i}_\rho$ for
    $i \in \{ 0, \ldots, p \} \setminus \{ p-1 \}$ and $M^{p, 0} \circ \, ^*$.
	\end{itemize}
\end{theorem}
\begin{proof}
	The map $\mathcal S \colon \ValCon{p}{2} \to \ValCo{p}{2}$ defined by
	\[
		\mathcal S \mu = \rho \cdot \mu
	\]
	is an isomorphism. The result now follows directly from Theorem \ref{th: 2-dim, SL(2)-covariant}.
\end{proof}

%------------------------------------------------------------------------------%
\subsection{The $n$-dimensional case}
%------------------------------------------------------------------------------%

In this section we will prove Theorem \ref{th: n-dim, SL(n)-covariant} by induction
over the dimension.
During the induction step, we will encounter a tensor valuation that might not be symmetric.
This makes it necessary to establish some results for non-symmetric tensor valuations first.
Note that the definition of $\SL$-contravariance given at the end of the previous section
extends in an obvious way to maps $\mu \colon \cP_o^n \to \Ten{\R^n}$.
Also recall that by our notation convention we set $J = [-c e_n, d e_n]$.

\begin{lemma}\label{le: f^B, SL(n)-contravariant}
	Let $n \geq 2$ and $\mu \colon \cP^n_o \to \Ten{\R^n}$ be a measurable $\SL$-contravariant valuation.
	If $\mu$ satisfies
	\begin{equation}\label{eq: mu[B, J]=0}
		\mu[B, J] = 0
	\end{equation}
	for all $B \in \cP^{n-1}_o$ and $c, d > 0$,
	then there exists a family of measurable functions $F^B \colon (\R^{n-1})^2 \to \Ten{\R^n}$ with
	\begin{equation*}
		F^B(x, y) = \mu \left[ B, -c \ptwo x 1, d \ptwo y 1 \right]
	\end{equation*}
	for all $B$, $c, d$ and $x, y \in \R^{n-1}$ which form a double pyramid.
	Furthermore, each $F^B$ satisfies
	\begin{equation}\label{eq: f^B(x, y), SL(n)-contravariant}
		F^B(x, y) = \ptwotwo {\Id} {-x^t} 0 1 \cdot F^B(0, y-x)
	\end{equation}
	and
	\begin{equation}\label{eq: f^B(0, x) linear, SL(n)-contravariant}
		F^B(0, x+y) = F^B(0, x) + \ptwotwo {\Id} {-x^t} 0 1 \cdot F^B(0, y) 
	\end{equation}
	for all $x, y \in \R^{n-1}$.
\end{lemma}

\begin{proof}
	The arguments which will be used are similar to the ones in the proof of Lemma \ref{le: existence F^I}.
	Let $B \in \cP_o^{n-1}$ and $x, y \in \R^{n-1}$ be given. 
	Choose $c, d > 0$ such that $B, c, d, x, y$ form a double pyramid.
	For sufficiently small $r > 0$ the valuation property implies
	\begin{multline*}
		\mu \left[ B, -c \ptwo x 1, d \ptwo y 1 \right] + \mu \left[ B, -r \ptwo y 1, r \ptwo y 1 \right] = \\
		\mu \left[ B, -c \ptwo x 1, r \ptwo y 1 \right] + \mu \left[ B, -r \ptwo y 1, d \ptwo y 1 \right] .
	\end{multline*}
	By assumption \eqref{eq: mu[B, J]=0} and the $\SL[2]$-contravariance of $\mu$ this simplifies to
	\[
		\mu \left[ B, -c \ptwo x 1, d \ptwo y 1 \right] = \mu \left[ B, -c \ptwo x 1, r \ptwo y 1 \right] .
	\]
	Hence, the expression
	\[
		\mu \left[ B, -c \ptwo x 1, d \ptwo y 1 \right]
	\]
	is independent of $d$.
	By analogous arguments we see that it is also independent of $c$.
	Therefore, the family $F^B$ is well defined.
	
	The $\SL$-contravariance of $\mu$ implies \eqref{eq: f^B(x, y), SL(n)-contravariant}.
	So it remains to prove \eqref{eq: f^B(0, x) linear, SL(n)-contravariant}.
	The valuation property of $\mu$ implies for sufficiently small $r > 0 $ that
	\[
		\mu \left[ B, -c \ptwo x 1, d \ptwo y 1 \right] + \mu \left[ B, -r e_n, r e_n \right] =
		\mu \left[ B, -c \ptwo x 1, r e_n \right] + \mu \left[ B, -r e_n, d \ptwo y 1 \right] .
	\]	
	By \eqref{eq: mu[B, J]=0} and the definition of $F^B$ we therefore obtain
	\[
		F^B(x, y) = F^B(x, 0) + F^B(0, y) .
	\]
	Combining this with \eqref{eq: f^B(x, y), SL(n)-contravariant} gives
	\[
		\ptwotwo {\Id} {-x^t} 0 1 \cdot F^B(0, y-x) = \ptwotwo {\Id} {-x^t} 0 1 \cdot F^B(0, -x) + F^B(0, y) .
	\]
	Replace $x$ by $-x$ in this equation.
	Then a matrix multiplication proves \eqref{eq: f^B(0, x) linear, SL(n)-contravariant}.
\end{proof}

Our next result deals with valuations which are not only compatible with the special linear group, but with $\PL$, i.e.\ 
linear maps with positive determinant.
We say that a map $\mu \colon \cP_o^n \to \Ten{\R^n}$ is $\PL$-contravariant
if there exists a $q \in \R$ such that
\[
	\mu(\phi P) = (\det \phi)^q \phi^{-t} \cdot \mu(P)
\]
for all $P \in \cP_o^n$ and each $\phi \in \PL$.
Clearly, every $\PL$-contravariant map is also $\SL$-contravariant.

\begin{theorem}\label{th: vanishes on crosspolytopes implies everywhere, GL(2)}
	Let $\mu \colon \cP^2_o \to \Ten{\R^2}$ be a measurable $\PL[2]$-contravariant valuation.
	If 
	\[
		\mu[I, J] = 0
	\]
	for all $a, b, c, d > 0$, then $\mu$ vanishes everywhere.
\end{theorem}
\begin{proof}
	By Theorem \ref{th: determined by values on SL(n)(R^n)} and the $\SL$-contravariance of $\mu$
	it is enough to show that $\mu$ vanishes on double pyramids.
	So let $a, b, c, d > 0$ and $x, y \in \R$ form a double pyramid.
	With the aid of \eqref{eq: integral compatible SL(n)}, we obtain from Lemmas 
	\ref{le: f^B, SL(n)-contravariant} and \ref{le: tensor Cauchy, contravariant} that
	\begin{equation}\label{eq: mu = int K contra}
		\mu \left[ I, -c \ptwo x 1, d \ptwo y 1 \right] = \int_x^y \ptwotwo 1 {-z} 0 1 \cdot K^I \ dz
	\end{equation}
	for some tensor $K^I \in \Ten{\R^2}$.
	We have to show that $K^I = 0$.

	Since $\mu$ is $\PL[2]$-contravariant, there exists a $q \in \R$ with
	\[
		\mu \left[ r I, -c \ptwo x 1, d \ptwo y 1 \right] =
		r^{2q-p} \mu \left[ I, - \frac c r \ptwo x 1, \frac d r \ptwo y 1 \right] =
		r^{2q-p} \int_x^y \ptwotwo 1 {-z} 0 1 \cdot K^I \ dz
	\]
	for all $r > 0$.
	Consequently, 
	\[
		\int_x^y \ptwotwo 1 {-z} 0 1 \cdot K^{rI} \ dz
		= r^{2q-p} \int_x^y \ptwotwo 1 {-z} 0 1 \cdot K^I \ dz .
	\]
	From \eqref{eq: integral tilt kernel transpose} we deduce that $I \mapsto K^I_\alpha$ is $(2q-p)$-homogeneous
	for all $\alpha \in \{ 1, 2 \}^p$.
	Similarly, 
	\[
		\mu \left[ - I, -c \ptwo x 1, d \ptwo y 1 \right] =
		(-1)^p \mu \left[ I, - d \ptwo y 1, c \ptwo x 1 \right] =
		(-1)^{p+1} \int_x^y \ptwotwo 1 {-z} 0 1 \cdot K^I \ dz .
	\]
	As before, we conclude that $I \mapsto K^I_\alpha$ is either even or odd.
	Since $\mu$ is a valuation, so is $I \mapsto K^I_\alpha$.
	In fact, this follows from \eqref{eq: mu = int K contra} and \eqref{eq: integral tilt kernel transpose}.
	By Theorems \ref{th: dim 1 even homogeneous} and \ref{th: dim 1 odd homogeneous} we therefore have
	\[
		K^I_\alpha = 
		\begin{cases}
			k_\alpha \left[ a^{2q-p} + (-1)^{p+1} b^{2q-p} \right] & \text{for $2q-p \neq 0$} , \\
			k_\alpha \left[ \ln(a) - \ln(b) \right] & \text{for $2q-p = 0$ and $p$ even} , \\
			k_\alpha & \text{for $2q-p = 0$ and $p$ odd},
		\end{cases}
	\]
	where $k_\alpha \in \R$ is some constant.
	Define a tensor $K \in \Ten{\R^2}$ componentwise by $K_{\alpha} = k_{\alpha}$.
	It remains to prove that $K$ vanishes.

	Let $s, u > 0$.
	Consider the triangle $T$ with corners at $s e_1 + e_2$, $-u e_1$ and $-e_2$.
	A simple calculation shows that $T$ can be written in two different ways, namely
	\[
		T = \left[ -u e_1, \frac s 2 e_1, -e_2, \ptwo s 1 \right]
	\]
	and
	\[
		T = \ptwotwo 0 1 {-1} 0 \left[ -e_1, \frac u {s+u} e_1, -s \ptwo {- \frac 1 s} 1, u e_2 \right] .
	\]
	By the $\SL[2]$-contravariance of $\mu$ and \eqref{eq: mu = int K contra} we get
	\begin{equation}\label{eq: int K}
		\int_0^s \ptwotwo 1 {-z} 0 1 \cdot K^{\left[ -u e_1, \frac s 2 e_1 \right]} \ dz =
		\ptwotwo 0 1 {-1} 0 \cdot \int_{- \frac 1 s}^0 \ptwotwo 1 {-z} 0 1 \cdot K^{\left[ -e_1, \frac u {s+u} e_1 \right]} \ dz .
	\end{equation}
	For the tensor $K$ from above, define two tensor polynomials by
	\[
		P(s) = \int_0^s \ptwotwo 1 {-z} 0 1 \cdot K \ dz \quad \text{and} \quad
		Q(s) = \ptwotwo 0 1 {-1} 0 \cdot \int_{-s}^0 \ptwotwo 1 {-z} 0 1 \cdot K \ dz .
	\]
	Recall that we have to prove $K = 0$.
	So by \eqref{eq: integral tilt kernel transpose} it is enough to show that either $P(s)$ or $Q(s)$ vanishes.
	
	Let $\alpha \in \{0, 1\}^p$. 
	It suffices to show $P_{\alpha}(s) = 0$ or $Q_{\alpha}(\frac 1 s) = 0$ for all positive $s$.
	First, assume $2q-p \neq 0$. 
	By the representation of $K^I$, the $\alpha$-component of equation \eqref{eq: int K} simplifies to
	\[
		\left( u^{2q-p} + (-1)^{p+1} \left( \frac s 2 \right)^{2q-p} \right) P_{\alpha}(s) =
		\left( 1 + (-1)^{p+1} \left( \frac u {s+u} \right)^{2q-p} \right) Q_{\alpha}(\tfrac 1 s) .
	\]	
	If $2q-p > 0$, let $u$ tend to infinity. 
	Note that the limit of the right hand side exists and is finite. 
	Thus, $P_{\alpha}(s) = 0$.
	Next, suppose that $2q-p < 0$. The last equation is clearly equivalent to
	\begin{equation}\label{eq: 2q-p < 0 P Q} 
		u^{2q-p}\left( P_{\alpha}(s) + (-1)^p (s+u)^{p-2q} Q_{\alpha}(\tfrac 1 s) \right) = 
		 (-1)^p \left( \frac s 2 \right)^{2q-p} P_{\alpha}(s) + Q_{\alpha}(\tfrac 1 s) .
	\end{equation}
	Let $u$ tend to $0$. Since the right hand side is constant in $u$, we must have
	\[
		Q_{\alpha}(\tfrac 1 s) = (-1)^{p+1} s^{2q-p} P_{\alpha}(s) .
	\]	
	If we set $u = st$ in \eqref{eq: 2q-p < 0 P Q} and use the last relation we obtain
	\[
		P_{\alpha}(s) \left( t^{2q-p} - (t+1)^{p-2q} + (-1)^p ( 1 - 2^{p-2q} ) \right) = 0
	\]
	for all $t > 0$. 
	Obviously, this implies $P_{\alpha}(s) = 0$.
	
	Second, let $2q-p = 0$ and $p$ be even.
	In this case, the $\alpha$-component of \eqref{eq: int K} is
	\[
		\Big( \ln(u) - \ln (\tfrac s 2 ) \Big) P_{\alpha}(s) = 
		-\ln\left( \frac {u} {s+u} \right)Q_{\alpha}(\tfrac 1 s) .
	\]
	The limit $u \to \infty$ implies that $P_{\alpha}(s) = 0$.
	
  Finally, let $2q-p = 0$ and $p$ be odd.	
  Then we have
  \[
		P_{\alpha}(s) = Q_{\alpha}(\tfrac 1 s) .
	\]
	The left hand side is a polynomial in $s$ without constant term and
	the right hand side is a polynomial in $\frac 1 s$ without constant term.
	Clearly, both polynomials have to be zero.
\end{proof}

Inductively, we will now extend this result to arbitrary dimensions.
Let us collect some notation before stating the next theorem.
In our context, an $n$-dimensional cross polytope is the convex hull of $n$ line segments
\[
	[-a_i e_i, b_i e_i], \qquad i = 1, \ldots, n,
\]
where, for all $i$,  the numbers $a_i$ and $b_i$ are positive.

Let $\alpha \in \{ 0, 1 \}^p$. We define a subspace $U_\alpha$ of $\Ten{\R^n}$ by
\[
	U_{\alpha} = \spn \left \{ x_1 \otimes \cdots \otimes x_p: 
	\text{$x_i = e_n$ if $\alpha_i = 1$, and $x_i \in \R^{n-1} \times \{0\}$ if $\alpha_i = 0$ } \right \} .
\]
In other words, the multiindex $\alpha$ indicates the positions of $e_n$ in a tensor product.
Note that
\[
	\Ten{\R^n} = \bigoplus_{\alpha \in \{ 0, 1 \}^p} U_\alpha .
\]
For $C \in \Ten{\R^n}$ we denote by $C_\alpha$ the projection of $C$ onto $U_\alpha$.
If $i$ is the total number of indices in $\alpha$ which are equal to one,
then $C_\alpha$ will be viewed as an element of $\Ten[p-i]{\R^{n-1}}$.

\begin{theorem}\label{th: vanishes on crosspolytopes implies everywhere, GL(n)}
	Let $n \geq 2$ and $\mu \colon \cP^n_o \to \Ten{\R^n}$ be a measurable $\PL$-contravariant valuation.
	If $\mu$ vanishes on crosspolytopes, then it vanishes everywhere.
\end{theorem}
\begin{proof}
	We will prove the theorem by induction.
	The case $n = 2$ is just a reformulation of
	Theorem \ref{th: vanishes on crosspolytopes implies everywhere, GL(2)}.
	So let $n \geq 3$ and assume that the theorem holds in dimension $n-1$.
	Fix numbers $c, d > 0$.
	For $\alpha \in \{ 0, 1 \}^p$ set
	\[
		\nu(B) = \mu_\alpha [B, J], \quad B \in \cP_o^{n-1} .
	\]
	Since $B \mapsto \nu(B)$ satisfies the induction assumption, it vanishes everywhere.
	Hence,
	\begin{equation}\label{eq: mu[B, J] = 0, GL(n)}
		\mu[B, J] = 0
	\end{equation}
	for all $B \in \cP^{n-1}$ and $c, d > 0$.
	From Lemma \ref{le: f^B, SL(n)-contravariant} we therefore obtain a family of functions $F^B$ such that
	\[
		F^B(x,y) = \mu \left[ B, -c \ptwo x 1, d \ptwo y 1 \right]
	\]
	for all $x, y \in \R^{n-1}$ whenever $B, c, d, x, y$ form a double pyramid.
	Set $G^B(x) = F^B(0, x)$.
	Next, we deduce two properties of $G^B$.
	First, equation \eqref{eq: f^B(0, x) linear, SL(n)-contravariant} becomes
	\begin{equation}\label{eq: G^B linear}
		G^B(x+y) = G^B(x) + \ptwotwo {\Id} {-x^t} 0 1 \cdot G^B(y) . 
	\end{equation}	
	By the $\PL$-contravariance of $\mu$, there exists a $q \in \R$ such that
	\[
		\mu \left[ \phi B, -c \ptwo x 1, d \ptwo y 1 \right] =
		(\det \phi)^q \diag {\phi^{-t}} 1 \cdot
		\mu \left[ B,
		-c \ptwo {\phi^{-1} x} 1,
		d \ptwo {\phi^{-1} y} 1
		\right]
	\]
	for all $\phi \in \PL[n-1]$.
	Therefore,
	\[
		G^{\phi B}(x) =
		(\det \phi)^q \diag {\phi^{-t}} 1 \cdot G^B( \phi^{-1} x ) .
	\]
	Projecting onto the subspace $U_\alpha$, $\alpha \in \{ 0, 1 \}$, immediately proves
	\begin{equation}\label{eq: g^{vartheta B}, GL(n)}
		G^{\phi B}_\alpha(x) = (\det \phi)^q \phi^{-t} \cdot G^B_\alpha ( \phi^{-1} x ) .
	\end{equation}
	Note that by \eqref{eq: f^B(x, y), SL(n)-contravariant} and Theorem \ref{th: determined by values on SL(n)(R^n)}
	it is enough to prove $G^B = 0$ for all $B \in \cP_o^{n-1}$.	
	
	We will show by induction that $G^B_\alpha = 0$.
	Assume that $G^B_\beta = 0$ for all $\beta \in \{ 0, 1 \}^p$ with $\beta < \alpha$,
	which is trivially true for $\alpha = (0, \ldots, 0)$.
	Equation \eqref{eq: G^B linear} and the induction assumption imply
	\[
		G^B_{\alpha}(x+y) = G^B_{\alpha}(x) + G^B_{\alpha}(y) .
	\]
	Clearly, the map $x \mapsto G^B_\alpha(x)$ is measurable.
	If $i$ denotes the number of ones in $\alpha$, then \eqref{eq: tensor cauchy} 
	implies that $G^B_\alpha$ can be viewed as an element of $\Ten[p-i+1]{\R^{n-1}}$, say $\tilde G^B$.
	Equation \eqref{eq: g^{vartheta B}, GL(n)} implies that
	\[
		\tilde G^{\phi B} = (\det \phi)^q \phi^{-t} \cdot \tilde G^B.
	\]
	Hence, $B \mapsto \tilde G^B$ is $\GL[n-1]$-contravariant.
	It is also a measurable valuation because $\mu$ has these properties.
	If we can show that this map vanishes on crosspolytopes, then we can apply our initial induction assumption 
	and the proof is completed.
	
	So let $B \in \cP_o^{n-1}$ be a crosspolytope and fix some $j \in \{ 1, \ldots, n-1 \}$.
	Since $n \geq 3$, we can choose a coordinate $k \in \{1, \ldots, n-1 \} \setminus \{ j \}$.
	Let $\phi \in \SL$ be the map with $e_k \mapsto e_n$ and $e_n \mapsto -e_k$ such that all 
	other canonical basis vectors stay fixed.
	Note that since $B$ is a crosspolytope, there exists a $\tilde B \in \cP_o^{n-1}$ and a line segment $\tilde J$ in the
	span of $e_n$ with
	\[
		\phi \left[ B, -c e_n, d (e_j + e_n) \right] = [\tilde B, \tilde J] .
	\]
	From the definition of $G^B$, the $\SL$-contravariance of $\mu$, and \eqref{eq: mu[B, J] = 0, GL(n)},
	it follows that $G^B(e_j) = 0$.
	In particular, also $G_\alpha^B(e_j) = 0$ and since $x \mapsto G^B_\alpha(x)$ is linear, we conclude that $\tilde G^B = 0$.
\end{proof}

Let us now come back to the symmetric setting.
For $i \in \{ 0, \ldots, p \}$ define subspaces $U_i$ of $\Sym{\R^n}$ by
\[
	U_i = \spn \left \{ x_1 \odot \cdots \odot x_{p-i} \odot e_n^{\odot i} :
	x_1, \ldots, x_{p-i} \in \R^{n-1} \times \{0\} \right \}.
\]
As before, $\Sym{\R^n}$ is the direct sum of these subspaces, i.e.
\[
	\Sym{\R^n} = \bigoplus_{i = 0}^p U_i .
\]
For $C \in \Sym{\R^n}$ we denote by $C_i$ the projection of $C$ onto $U_i$,
and $C_i$ will be viewed as an element of $\Sym[p-i]{\R^{n-1}}$.
We remark that for the planar case $\Sym{\R^2}$, this notation coincides with the one for tensor components used before.

\begin{lemma}\label{le: vanishes on crosspolytopes implies everywhere, SL(n)}
	Let $n \geq 2$ and $\mu \in \ValCon{p}{n}$.
	If $\mu$ vanishes on all crosspolytopes, then it vanishes everywhere.
\end{lemma}
\begin{proof}
	Let $n = 2$. 
	By Theorem \ref{th: 2-dim, SL(2)-contravariant}, we can write $\mu$ as a linear combination 
	\[
		\mu(P) = \sum_{\stackrel{i=0}{i \neq p-1}}^p c_i \rho \cdot M_\rho^{i,p-i}(P)
		+ c_{p+1} M^{p,0}(P^*) , \quad P \in \cP^n_o .
	\]
	The assumption that $\mu$ vanishes on crosspolytopes yields
	\[
		\sum_{\stackrel{i=0}{i \neq p-1}}^p c_i \left[ \rho \cdot M_\rho^{i,p-i}(B) \right]_p
		+ c_{p+1} \left[ M^{p,0}(B^*) \right]_p = 0
	\]
	for all crosspolytopes $B$.
	By \eqref{eq: polar homogeneity} and \eqref{eq: M homogeneous 2 dim} all these operators have different degrees of homogeneity.
	Hence, we can compare coefficients.
	Since by \eqref{eq: M(B) neq 0},
	\[
		\left[ \rho \cdot M_\rho^{i,p-i}(B)\right]_p \quad \text{and} \quad
		\left[ M_\rho^{p,0}(B^*)\right]_p
	\]	
	do not vanish for all crosspolytopes $B$,
	all $c_i$ have to be zero, which in turn settles the case $n=2$.

	Let $n \geq 3$ and assume that the theorem holds in dimension $n-1$.
	Exactly as in the beginning of the proof of Theorem \ref{th: vanishes on crosspolytopes implies everywhere, GL(n)} we obtain
	\[
		\mu[B, J] = 0
	\]
	for all $B \in \cP^{n-1}_o$ and $c, d > 0$.
	
	Again, we can apply Lemma \ref{le: f^B, SL(n)-contravariant} to obtain a family of functions $F^B$ such that
	\[
		F^B(x,y) = \mu \left[ B, -c \ptwo x 1, d \ptwo y 1 \right]
	\]
	for all $x, y \in \R^{n-1}$ whenever $B, c, d, x, y$ form a double pyramid.
	Set $G^B(x) = F^B(0, x)$.
	As before, we now deduce some properties of $G^B$.
	First, equation \eqref{eq: f^B(0, x) linear, SL(n)-contravariant} becomes
	\begin{equation}\label{eq: G^B linear symm}
		G^B(x+y) = G^B(x) + \ptwotwo {\Id} {-x^t} 0 1 \cdot G^B(y) . 
	\end{equation}
	By the $\SL$-contravariance of $\mu$ we have
	\[
		\mu \left[ \phi B, -c \ptwo x 1, d \ptwo y 1 \right] =
		\diag {\phi^{-t}} {\det \phi} \cdot
		\mu \left[ B,
		-c \det \phi \ptwo {\frac 1 {\det \phi} \phi^{-1} x} 1,
		d \det \phi \ptwo {\frac 1 {\det \phi} \phi^{-1} y} 1
		\right]
	\]
	for all $\phi \in \PL[n-1]$.
	By the definition of $G^B$ we therefore get
	\[
		G^{\phi B}(x) =
		\diag {\phi^{-t}} {\det \phi} \cdot G^B \left( \frac 1 {\det \phi} \phi^{-1} x \right) .
	\]
	In particular,
	\begin{equation}\label{eq: g^{vartheta B}, SL(n)}
		G^{\phi B}_i(x) =
		(\det \phi)^i \phi^{-t} \cdot G_i^B \left( \frac 1 {\det \phi} \phi^{-1} x \right) .
	\end{equation}
	By \eqref{eq: f^B(x, y), SL(n)-contravariant} and Theorem \ref{th: determined by values on SL(n)(R^n)}
	it suffices again to prove $G^B = 0$ for all $B \in \cP_o^{n-1}$.
	
	We will show by induction that $G^B_i = 0$, $i \in \{ 0, \ldots, p\}$.
	Assume that $G^B_j = 0$ for all $j \in \{ 0, \ldots, p \}$ with $j < i$,
	which is trivially true for $i = 0$.
	Equation \eqref{eq: G^B linear symm} together with the induction assumption proves
	\[
		G_i^B(x+y) = G_i^B(x) + G_i^B(y).
	\]
	Since $x \mapsto G_i^B(x)$ is also measurable, it can be interpreted as an element of $\Ten[p-i+1]{\R^{n-1}}$, say $\tilde G^B$.
	Note that $\tilde G^B$ need no longer be symmetric.
	Equation \eqref{eq: g^{vartheta B}, SL(n)} implies
	\[
		\tilde G^{\phi B} = (\det \phi)^{i-1} \phi^{-t} \cdot \tilde G^B 
	\]
	for each $\phi \in \PL[n-1]$.
	Thus, $B \mapsto \tilde G^B$ is a measurable $\PL[n-1]$-contravariant valuation.
	With precisely the same argument as at the end of the proof of Theorem \ref{th: vanishes on crosspolytopes implies everywhere, GL(n)}
	it follows that $\tilde G^B$ vanishes for crosspolytopes.
	Our initial induction assumption then implies that $\tilde G^B = 0$, which in turn yields $G^B_i = 0$.
\end{proof}

Before we continue, let us collect some notation.
For $a_1, b_1, \ldots, a_{n-1}, b_{n-1} > 0$ and $c, d > 0$ define $n$ line segments by
\begin{equation}\label{eq: def intervals}
	I_1 = [-a_1 e_1, b_1 e_1], \ldots, I_{n-1} = [-a_{n-1} e_{n-1}, b_{n-1} e_{n-1}] \quad \text{and} \quad J = [-c e_n, d e_n] .
\end{equation}
Furthermore, set
\[
	\tilde I_{n-1} = [-a_{n-1} e_n, b_{n-1} e_n] \quad \text{and} \quad \tilde J = [-c e_{n-1}, d e_{n-1}] .
\]
Finally, define $B = [ I_1, \ldots, I_{n-2}, I_{n-1} ]$ and $\tilde B = [ I_1, \ldots, I_{n-2}, - \tilde J ]$.

\begin{theorem}\label{th: n-dim, SL(n)-contravariant}
	For $n \geq 3$ the following holds.
	\begin{itemize}
		\item
		A basis of $\ValCon{0}{n}$ is given by $\chi$, $V$ and $V \circ \, ^*$.

		\item
    A basis of $\ValCon{1}{n}$ is given by $M^{p, 0} \circ \, ^*$.

		\item
    For $p \geq 2$, a basis of $\ValCon{p}{n}$ is given by $M^{p, 0} \circ \, ^*$ and $M^{0, p} $.
	\end{itemize}
\end{theorem}

\begin{proof}
	We already know three, one, and two linearly independent elements of 
	$\ValCon{0}{n}$, $\ValCon{1}{n}$, and $\ValCon{p}{n}$ for $p \geq 2$, respectively.
	By Lemma \ref{le: vanishes on crosspolytopes implies everywhere, SL(n)} and
	Theorem \ref{th: determined by values on SL(n)(R^n)} it is enough to prove
	that if we restrict the maps in $\ValCon{p}{n}$, $p \geq 0$, to crosspolytopes,
	the resulting spaces are at most three-, one-, and two-dimensional, respectively.

	We will prove this by induction over the dimension.
	Before we do so, let us collect some prerequisites.
	For $i \in \{ 0, \ldots, p \}$ and fixed positive numbers $c$ and $d$ consider the map
	\begin{equation}\label{eq: def mu_i}
		B \mapsto \mu_i [B, J] , \quad B \in \cP_o^{n-1} ,
	\end{equation}
	where $J := [-c e_n, d e_n]$.
	By the $\SL$-contravariance of $\mu$ we obtain
	\[
		\mu[B, r J] = \diag {r^{\frac 1 {n-1}} \Id} {\frac 1 r} \cdot \mu[r^{\frac 1 {n-1}} B, J]
	\]
	for all $r > 0$ and
	\[
		\mu[B, - J] = \diag {\vartheta^t} {-1} \cdot \mu[\vartheta B, J]
	\]
	for every $\vartheta \in \VL$ with $\det \vartheta = -1$.	By projecting onto the subspace $U_i$ we get
	\begin{equation}\label{eq: mu_i homogeneous}
		\mu_i[B, r J] = r^{ \frac{p-i}{n-1} - i} \mu_i[r^{\frac{1}{n-1}} B, J]
	\end{equation}
	and
	\begin{equation}\label{eq: mu_i reflection}
		\mu_i[B, - J] = (-1)^i \vartheta^t \cdot \mu_i[\vartheta B, J] .
	\end{equation}	
	Define intervals as in \eqref{eq: def intervals}.
	By the $\SL$-contravariance of $\mu$ again we have
	\[
		\mu \left[ I_1, \ldots, I_{n-2}, I_{n-1}, \lambda J \right] = 
		\diag {\Id}{\begin{array}{cc} 0 & \lambda \\ - \frac 1 \lambda & 0 \end{array}} \cdot
		\mu \left[ I_1, \ldots, I_{n-2}, - \tilde J, \lambda \tilde I_{n-1} \right]
	\]
	for all $\lambda > 0$.
	With the above definitions of $B$ and $\tilde B$ we therefore get
	\begin{equation}\label{eq: comparing homogeneity grades}
		\left( \mu_0 \right)_i [B, \lambda J ] =
		\lambda^i \left( \mu_i \right)_0 [ \tilde B, \lambda \tilde I_{n-1} ] .
	\end{equation}
	Here, the second indices denote projections in $\Sym{\R^{n-1}}$ and $\Sym[p-i]{\R^{n-1}}$, respectively.
	In other words, on the left hand side we have the component with $0$ times $e_n$ and $i$ times $e_{n-1}$,
	and on the right hand side the component with $i$ times $e_n$ and $0$ times $e_{n-1}$.

	Now, we can start with our induction.
	Let $n = 3$.
	From the $\SL[3]$-contravariance of $\mu$, it follows that the map \eqref{eq: def mu_i} is $\SL[2]$-contravariant.
	By Theorem \ref{th: 2-dim, SL(2)-contravariant} and the convention that $\mu_i$ will be viewed 
	as an element of $\Ten[p-i]{\R^2}$, we therefore have
	\begin{equation}\label{eq: mu_i sum}
		\mu_i [B, J] =
		l_i^J M_{\rho}^{p-i, 0}(B^*) + \sum_{j=0}^{p-i} k_{i,j}^J \, \rho \cdot M_{\rho}^{j, p-i-j}(B)
	\end{equation}
	for $i \neq p$ and
	\begin{equation}\label{eq: mu_p sum}
		\mu_p [B, J] = l_p^J V(B^*) + k_{p,0}^J V(B) + m_p^J ,
	\end{equation}
	for $i = p$, where $k_{i,j}^J, l_i^J, m_p^J \in \R$ and $k_{i, p-i-1}^J = 0$.
	Note that the operators on the right hand side are now operators in dimension $2$.

	Let $i \neq p$.
	If we plug \eqref{eq: mu_i sum} into \eqref{eq: mu_i homogeneous} and use the appropriate degrees of homogeneity
	from \eqref{eq: M homogeneous 2 dim} and \eqref{eq: polar homogeneity}, we obtain
	\begin{multline*}
		l_i^{rJ} M_{\rho}^{p-i, 0}(B^*) + \sum_{j=0}^{p-i} k_{i,j}^{rJ} \, \rho \cdot M_{\rho}^{j, p-i-j}(B) =\\
		{} r^{-1-i}l_i^J M_{\rho}^{p-i, 0}(B^*) + \sum_{j=0}^{p-i} r^{1-i+j}k_{i,j}^J \, \rho \cdot M_{\rho}^{j, p-i-j}(B) .
	\end{multline*}
	As in the proof of Lemma \ref{le: vanishes on crosspolytopes implies everywhere, SL(n)},
	we can compare coefficients in this equation.
	This shows that $J \mapsto l_i^J$ and $J \mapsto k_{i,j}^J$ are homogeneous.
	For $i = p$ we can argue similarly.
	We conclude that $J \mapsto l_i^J$, $J \mapsto k_{i,j}^J$ and $J \mapsto m_p^J$ have degrees of homogeneity
	$-1-i$, $1-i+j$, and $-p$, respectively.
	
	With the same procedure, we obtain by \eqref{eq: mu_i reflection} and \eqref{eq: M_rho signum covariant} that
	\begin{equation}\label{eq: l k even odd}
		l_i^{-J} = (-1)^i l_i^J , \quad k_{i,j}^{-J} = (-1)^{p+j} k_{i,j}^J \quad \text{and} \quad m_p^{-J} = (-1)^i m_p^J .
	\end{equation}
	In particular, the maps $J \mapsto l_i^J$, $J \mapsto k_{i,j}^J$ and $J \mapsto m_p^J$ are even or odd. 
	Comparing coefficients again, also proves that these maps
	$l_i^J$, $k_{i,j}^J$ and $m_p^J$ are measurable valuations with respect to $J$. 
	From Theorems \ref{th: dim 1 even homogeneous} and \ref{th: dim 1 odd homogeneous}
	we deduce that $l_i^J$, $k_{i,j}^J$ and $m_p^J$ are determined by constants
	$l_i \in \R$, $k_{i,j} \in \R$ and $m_p \in \R$, respectively.

	For $p = 0$, we are already done by \eqref{eq: mu_p sum}.
	So assume $p \geq 1$.
	If we plug representation \eqref{eq: mu_i sum} and \eqref{eq: mu_p sum}
	into \eqref{eq: comparing homogeneity grades} for $i = p$ and use the homogeneity of
	$k_{i,j}^J$, $l_i^J$ and $m_p^J$ with respect to $J$, we obtain
 	\[
		\lambda^{-1} l_0^{J} \left[ M_{\rho}^{p, 0}(B^*) \right]_p + 
		\sum_{j=0}^{p} \lambda^{1+j} k_{0,j}^{J} \left[  \rho \cdot M_{\rho}^{j, p-j}(B) \right]_p =
		\lambda^{-1}l_p^{\tilde I_2} V(\tilde B^*) + \lambda k_{p,0}^{\tilde I_2} V(\tilde B) + m_p^{\tilde I_2} .
	\]
	Therefore, $m_p = 0$ and $k_{0,j} = 0$, $j \neq 0$.
	Furthermore, $k_{p,0}$ is a multiple of $k_{0,0}$ and $l_p$ is a multiple of $l_0$.

	In the case $p = 1$ we know $k_{0,0} = 0$ from \eqref{eq: mu_i sum} and we are done.
	Assume $p \geq 2$.
	We already know that
	\[
		\mu_0 [B, J] = l_0^J M^{p, 0}(B^*) + k_{0,0}^J M^{0, p}(B).
	\]
	If we plug this and representation \eqref{eq: mu_i sum} into \eqref{eq: comparing homogeneity grades} for $i \neq p$ and use the homogeneity of
	$k_{i,j}^J$, $l_i^J$ and $m_p^J$ with respect to $J$, we obtain
	\begin{multline*}
		\lambda^{-1} l_0^{J} \left[ M_{\rho}^{p,0}(B^*) \right]_i + 
		\lambda k_{0,0}^J \left[ M_{\rho}^{0, p}(B) \right]_i  = \\
		\lambda^{-1} l_i^{\tilde I_2} \left[ M_{\rho}^{p-i,0}(\tilde B^*) \right]_0+
		\sum_{j=0}^{p-i} \lambda^{1+j} k_{i,j}^{\tilde I_2} \left[ \rho \cdot M_{\rho}^{j, p-i-j}(\tilde B) \right]_0 .
	\end{multline*}
	Therefore, $k_{i,j} = 0$, $j \neq 0$.
	Furthermore, $k_{i,0}$ is a multiple of $k_{0,0}$ and $l_i$ is a multiple of $l_0$.
	For the preceding argument, note that
	$\left[ M^{j, p-i-j}_\rho(\tilde B) \right]_0$, $j \in \{ 0, \ldots, p-i \} \setminus \{ p-i-1 \}$, and $\left[ M^{p-i, 0}_\rho(\tilde B^*) \right]_0$ do not vanish
	for all choices of intervals by \eqref{eq: M(B) neq 0} and the $\SL[2]$-contravariance of these operators.
	This completes the proof for $n = 3$.

	Next, assume $n > 3$ and that the theorem already holds in dimension $n-1$.
	From the $\SL$-contravariance of $\mu$, it follows that the map \eqref{eq: def mu_i} is $\SL[n-1]$-contravariant.
	By the induction assumption we have
	\begin{equation}\label{eq: mu_i sum n dim}
		\mu_i [B, J] =
		l_i^J M^{p-i, 0}(B^*) + k_i^J  M^{0, p-i}(B)
	\end{equation}
	for $i \neq p$ and
	\begin{equation}\label{eq: mu_p sum n dim}
		\mu_p [B, J] = l_p^J V(B^*) + k_p^J V(B) + m_p^J ,
	\end{equation}
	for $i = p$, where $l_i^J, k_i^J, m_p^J \in \R$ and $k_{p-1}^J = 0$.
	Note that the operators on the right hand side are now operators in dimension $n-1$.

	Let $i \neq p$.
	If we plug \eqref{eq: mu_i sum n dim} into \eqref{eq: mu_i homogeneous} and use the appropriate degrees of homogeneity
	from \eqref{eq: M homogeneous n dim} and \eqref{eq: polar homogeneity}, we obtain
	\[
		l_i^{rJ} M^{p-i, 0}(B^*) + k_i^{rJ}  M^{0, p-i}(B) =
		r^{-1-i}l_i^J M^{p-i, 0}(B^*) + r^{1-i}k_i^J  M^{0, p-i}(B) .
	\]
  As before, using \eqref{eq: M homogeneous n dim} and \eqref{eq: polar homogeneity}, we can compare coefficients.
	This shows that $J \mapsto l_i^J$ and $J \mapsto k_i^J$ are homogeneous.
	For $i = p$ we can argue similarly.
	We conclude that $J \mapsto l_i^J$, $J \mapsto k_i^J$ and $J \mapsto m_p^J$ have degrees of homogeneity
	$-1-i$, $1-i$, and $-p$, respectively.
	
	With the same procedure, we obtain by \eqref{eq: mu_i reflection} and \eqref{eq: M covariant, n dim} that
	\begin{equation}\label{eq: l k even odd n dim}
		l_i^{-J} = (-1)^i l_i^J , \quad k_i^{-J} = (-1)^i k_i^J \quad \text{and} \quad m_p^{-J} = (-1)^i m_p^J .
	\end{equation}
	In particular, the maps $J \mapsto l_i^J$, $J \mapsto k_i^J$ and $J \mapsto m_p^J$ are even or odd. 
  As before we can argue that these maps are measurable valuations. 
	From Theorems \ref{th: dim 1 even homogeneous} and \ref{th: dim 1 odd homogeneous}
	we deduce that $l_i^J$, $k_i^J$ and $m_p^J$ are determined by constants
	$l_i \in \R$, $k_i \in \R$ and $m_p \in \R$, respectively.

	For $p = 0$, we are already done by \eqref{eq: mu_p sum n dim}.
	So assume $p \geq 1$.
	If we plug representation \eqref{eq: mu_i sum n dim} and \eqref{eq: mu_p sum n dim}
	into \eqref{eq: comparing homogeneity grades} for $i = p$ and use the homogeneity of
	$l_i^J$, $k_i^J$ and $m_p^J$ with respect to $J$, we obtain
	\[
		\lambda^{-1} l_0^J \left[ M^{p,0}(B^*) \right]_p + \lambda k_0^J \left[ M^{0,p}(B) \right]_p =  
		\lambda^{-1} l_p^{\tilde I_{n-1}} V(\tilde B^*) + 
		\lambda k_p^{\tilde I_{n-1}} V(\tilde B) + m_p^{\tilde I_{n-1}} .
	\]
	Therefore, $m_p = 0$.
	Furthermore, $k_p$ is a multiple of $k_0$ and $l_p$ is a multiple of $l_0$.

	In the case $p = 1$ we know $k_0 = 0$ from \eqref{eq: mu_i sum n dim} and we are done.
	Assume $p \geq 2$.
	If we plug representation \eqref{eq: mu_i sum n dim} into \eqref{eq: comparing homogeneity grades} for $i \neq p$ and use the homogeneity of
	$l_i^J$ and $k_i^J$ with respect to $J$, we obtain
	\begin{multline*}
		\lambda^{-1} l_0^J \left[ M^{p,0}(B^*) \right]_i + \lambda k_0^J \left[ M^{0,p}(B) \right]_i = \\
		\lambda^{-1} l_i^{\tilde I_{n-1}} \left[ M^{p-i,0}(B^*) \right]_0 +
		\lambda k_i^{\tilde I_{n-1}} \left[ M^{0,p-i}(B) \right]_0 .
	\end{multline*}
	As before, using \eqref{eq: M(B) neq 0, n dim}, we conclude that $k_i$ is a multiple of $k_0$ and $l_i$ is a multiple of $l_0$.
\end{proof}

Finally, we prove our classification for $\ValCo{p}{n}$.

\begin{proof}
	[
		Proof of Theorem \ref{th: n-dim, SL(n)-covariant}
	]
	The map $\mathcal S \colon \ValCo{p}{n} \to \ValCon{p}{n}$ defined by
	\[
		\mathcal S \mu = \mu \circ \, ^*
	\]
	is an isomorphism. The result now follows directly from Theorem \ref{th: n-dim, SL(n)-contravariant}.
\end{proof}

	\section{Acknowledgements}
		%------------------------------------------------------------------------------%
The work of the second author was supported by the ETH Zurich Postdoctoral Fellowship Program
and the Marie Curie Actions for People COFUND Program.
The second author wants to especially thank Michael Eichmair and Horst Knörrer for being his mentors at ETH Zurich.

	\bigskip \goodbreak \noindent
	%------------------------------------------------------------------------------%
\begin{minipage}{0.63\textwidth}
	Christoph Haberl \\
	Vienna University of Technology \\
	Institute of Discrete Mathematics and Geometry \\
	Wiedner Hauptstraße 8--10/104 \\
	1040 Wien, Austria \\
	christoph.haberl@tuwien.ac.at
\end{minipage}
\begin{minipage}{0.35\textwidth}
	Lukas Parapatits \\
	ETH Zurich \\
	Department of Mathematics \\
	Rämistrasse 101 \\
	8092 Zürich, Switzerland \\
	lukas.parapatits@math.ethz.ch
\end{minipage}

\end{document}